\documentclass[11pt]{amsart}
\usepackage[left=1in,right=1in,top=1.25in,bottom=1.25in]{geometry}
\usepackage{mathtools, amssymb, amsfonts, cancel, array, dsfont}
\usepackage{multirow}
\usepackage{newtxtext}
\usepackage{xcolor}
\usepackage{hyperref}
\hypersetup{hidelinks}
\usepackage{pgfplots}
\usepgfplotslibrary{groupplots}
\pgfplotsset{compat=1.18}
\definecolor{inkmuted}{HTML}{52514E}
\definecolor{gridgray}{HTML}{E4E3DF}

\pgfplotsset{
  smlst/.style={color=blue, dashed, line width=1.5pt},
  smll/.style={color=red, dotted, line width=1.5pt},
  midl/.style={solid, line width=2.0pt},
  larg/.style={color=purple, dashdotted, line width=1.5pt},
  huge/.style={color=brown, solid, line width=1.5pt}
}
 
\newcommand{\datadir}{data}
 
\newcommand{\addrunM}[4]{%
  \addplot[#4]
    table[col sep=comma, x=t, y=E] {\datadir/torus_eta#1_sigma#2_M#3.csv};
  \addlegendentry{$M=#3$}%
}

\newcommand{\addruneta}[4]{%
  \addplot[#4]
    table[col sep=comma, x=t, y=E] {\datadir/torus_eta#1_sigma#2_M#3.csv};
  \addlegendentry{$\eta=#1$}%
}

\newtheorem{theorem}{Theorem}[section]
\newtheorem{theorem2}{Theorem}
\newtheorem{proposition}[theorem]{Proposition}
\newtheorem{lemma}[theorem]{Lemma}

\theoremstyle{definition}

\newtheorem{remark}[theorem]{Remark}
\newtheorem{assump}[theorem2]{Assumption}

\let\pa\partial

\newcommand{\cE}{\mathcal{E}}

\newcommand{\cP}{\mathcal{P}}
\newcommand{\cR}{\mathcal{R}}
\newcommand{\cT}{\mathcal{T}}

\newcommand{\ba}{\mathbf{a}}
\newcommand{\be}{\mathbf{e}}

\newcommand{\bn}{\mathbf{n}}
\newcommand{\bomega}{\boldsymbol{\omega}}
\newcommand{\bp}{\mathbf{p}}
\newcommand{\bphi}{\boldsymbol{\phi}}
\newcommand{\bvarphi}{\boldsymbol{\varphi}}
\newcommand{\br}{\mathbf{r}}
\newcommand{\bs}{\mathbf{s}}
\newcommand{\bu}{\mathbf{u}}
\newcommand{\bv}{\mathbf{v}}
\newcommand{\bw}{\mathbf{w}}
\newcommand{\bx}{\mathbf{x}}

\newcommand{\bz}{\mathbf{z}}
\newcommand{\bzero}{\mathbf{0}}
\newcommand{\bell}{\boldsymbol{\ell}}

\newcommand{\bA}{\mathbf{A}}
\newcommand{\bB}{\mathbf{B}}
\newcommand{\bC}{\mathbf{C}}
\newcommand{\bH}{\mathbf{H}}
\newcommand{\bI}{\mathbf{I}}
\newcommand{\bL}{\mathbf{L}}
\newcommand{\bM}{\mathbf{M}}
\newcommand{\bP}{\mathbf{P}}
\newcommand{\bR}{\mathbf{R}}

\newcommand{\bV}{\mathbf{V}}
\newcommand{\bW}{\mathbf{W}}

\newcommand{\dL}{\mathds{L}}

\newcommand{\rJ}{\mathrm{J}}

\newcommand{\uD}{\underline{D}}

\newcommand{\bbR}{\mathbb{R}}

\newcommand{\wt}{\widetilde}
\newcommand{\wh}{\widehat}

\newcommand{\oc}{\overline{c}}

\newcommand{\op}{\overline{p}}
\newcommand{\ophi}{\overline{\phi}}
\newcommand{\omu}{\overline{\mu}}
\newcommand{\obu}{\overline{\mathbf{u}}}
\newcommand{\oxi}{\overline{\xi}}
\newcommand{\oz}{\overline{z}}

\newcommand{\divG}{\mathrm{div}_\G}
\newcommand{\DeltaG}{\Delta_\G}

\newcommand{\G}{\Gamma}
\newcommand{\dst}{\operatorname{dst}}
\newcommand{\nablaG}{\nabla_\Gamma}

\renewcommand{\d}{\mathrm{d}}
\newcommand{\dt}{\,\mathrm{d}t}
\newcommand{\ds}{\,\mathrm{d}s}

\newcommand{\ddt}{\frac{\mathrm{\d}}{\mathrm{d}t}}
\newcommand{\mean}[1]{\langle #1 \rangle_\G}
\newcommand{\zm}[1]{#1^\#}
\newcommand{\tr}{\operatorname{tr}}

\newcommand{\dtau}{\delta_\tau}

\begin{document}

\title[Stokes--Cahn--Hilliard System: Analysis and Discretization]{Surface Stokes--Cahn--Hilliard System: Analysis and Structure-Preserving Discretization}

\author[M. A. Olshanskii]{Maxim Olshanskii}
\author[M. Shakipov]{Mansur Shakipov}
\address{Department of Mathematics\\ University of Houston\\ Houston, TX 77204}
\email{\{maolshanskiy,\,mshakipov\}@uh.edu}

\thanks{The authors were partially supported by {the} U.S. National Science Foundation under award DMS-2408978.}

\keywords{Surface Stokes--Cahn--Hilliard equation, well-posedness, continuous dependence on data, trace finite element method, stabilization, lateral phase separation}

\subjclass[2020]{Primary 35Q35, 35Q92, 65M60, 76D07; Secondary 35B30, 35D30, 35B65, 35R01, 76D03, 76T06}

\begin{abstract}
We study the surface Stokes--Cahn--Hilliard system governing tangential flow and phase separation in viscous fluid membranes. The equations are posed on a closed, stationary $C^2$ hypersurface $\Gamma\subset\mathbb{R}^{d+1}$, $d=2,3$. For constant viscosity and mobility and a class of regular double-well potentials, we prove global existence and uniqueness of weak solutions for initial phase fields in $H^1(\Gamma)$, with the velocity chosen orthogonal to the Killing fields. These solutions satisfy an energy equality and conserve mass. We establish Lipschitz continuous dependence in energy norms on every finite time interval, including for initial data with different means, and derive higher regularity results for $d=2$ under additional smoothness assumptions on the surface, potential, and initial data. We also develop a fully discrete unfitted finite element method that conserves the total phase and satisfies an exact discrete energy identity. The method combines a penalty on approximate Killing fields with projections in the transport and capillary terms to control these modes while preserving the energy balance.  Numerical experiments examine the identification of Killing fields and illustrate phase separation on a torus and a liposome with experimentally informed material parameters.
\end{abstract}

\maketitle

\section{Introduction}
\label{sec:introduction}

Interest in lateral phase separation in multicomponent lipid membranes is motivated by its connection to the lipid raft hypothesis for cell membranes~\cite{SimonsIkonen1997}, observations of coexisting fluid domains in multicomponent vesicles~\cite{VeatchKeller2003,BaumgartHessWebb2003}, and applications to the engineering of functional liposomes and lipid bilayers~\cite{membr1}.  Continuum models describe domain evolution at experimentally accessible scales by coupling diffuse-interface equations to membrane mechanics and hydrodynamics. Earlier models account for bending elasticity, phase separation, and viscous bulk flow~\cite{LowengrubRatzVoigt2009,SohnTsengLiVoigtLowengrub2010}; more recent studies address the interaction of surface flow, composition, and geometry in stationary and deformable membranes~\cite{BachiniKrauseVoigt2023,BachiniKrauseNitschkeVoigt2023}. In~\cite{WangPalzhanovQuainiOlshanskiiMajd2022lipid}, the authors found that simulations with a surface Navier--Stokes--Cahn--Hilliard model agree well with fluorescence microscopy measurements of domain coarsening for different membrane compositions. These developments motivate the analysis of coupled surface fluid and phase-field equations and numerical methods that respect their conservation and dissipation properties.

In this paper we study the surface Stokes--Cahn--Hilliard system on a closed, stationary hypersurface $\Gamma\subset\mathbb{R}^{d+1}$, $d=2,3$. The system couples incompressible tangential flow to phase separation through advection and capillary forces. It follows from the surface Navier--Stokes--Cahn--Hilliard model by neglecting inertia at small surface Reynolds numbers, an assumption consistent with the material parameters of multicomponent lipid bilayers~\cite{WangPalzhanovQuainiOlshanskiiMajd2022lipid}. Inertia-free surface fluid equations also arise in models coupling membrane flow to bending elasticity~\cite{ArroyoDeSimone2009,TorresSanchezMillanArroyo2019} and active geometric forces~\cite{PorrmannVoigt2024}. For constant viscosity and mobility and a regular double-well potential, we address the well-posedness of a weak formulation and regularity of weak solutions. We also develop a finite element method and study its conservation and stability properties.

There is extensive literature on the analysis of Navier--Stokes--Cahn--Hilliard (NSCH) systems in Euclidean domains; see, e.g.,~\cite{Abels2009,GiorginiMiranvilleTemam2019}. For prescribed evolving surfaces, Elliott and Sales~\cite{ElliottSales2025} proved well-posedness of weak solutions for matched densities and regular or logarithmic potentials. Abels, Garcke, and Poiatti~\cite{AbelsGarckePoiatti2026local,AbelsGarckePoiatti2025global} established local and global well-posedness of strong solutions for different densities and a logarithmic potential. Their continuous dependence estimate for the coupled system uses the higher regularity of strong solutions~\cite[Section~5]{AbelsGarckePoiatti2025global}. For two-dimensional bulk domains, Giorgini, Miranville, and Temam~\cite{GiorginiMiranvilleTemam2019} derive estimates in dual norms and recover continuous dependence of weak solutions in the energy space at positive times through regularization and interpolation. These results concern fluid systems with strictly positive densities. In particular, velocity estimates use the density-weighted kinetic energy and do not directly apply when the inertia coefficient vanishes. Our aim here is to establish, for both $d=2$ and $d=3$, well-posedness of the inertia-free problem and Lipschitz continuous dependence of weak solutions on the initial data in the energy norm, uniformly on every finite time interval $[0,T]$. For the regular potentials considered here, we control the nonlinear potential contribution using the regularity available for weak solutions, which allows us to obtain this estimate directly from the initial time without requiring stronger initial regularity.

We are also interested in relaxing the assumptions on surface smoothness. The weak surface NSCH theory in~\cite{ElliottSales2025} assumes $C^3$ geometry, and the strong theory in~\cite{AbelsGarckePoiatti2026local,AbelsGarckePoiatti2025global} uses $C^6$ geometry. For the scalar Cahn--Hilliard equation, a theory on $C^2$ evolving surfaces is available in~\cite{CaetanoElliott2021}. We assume $C^2$ surface smoothness for the well-posedness analysis of the coupled problem and account explicitly for the geometric assumptions needed for higher regularity of solutions. Keeping track of these assumptions is also relevant to extensions to actively deforming surfaces, where the smoothness of the geometry depends on the regularity of solutions and has to be established.

An additional issue is the possible presence of Killing vector fields, i.e., tangential fields with vanishing surface rate of strain. These fields form the kernel of the surface viscous operator, so the Stokes momentum equation determines the velocity only up to a Killing field. In the full NSCH system, the evolution of this component is governed by the momentum evolution equation and the initial velocity. To address this issue for the inertia-free problem, we show that the capillary force is orthogonal to the Killing fields, ensuring compatibility of the momentum equation, and select a unique velocity by requiring it to be $L^2(\Gamma)$-orthogonal to all Killing fields. This choice matters for the coupled system: adding a Killing field leaves the viscous stress unchanged but generally changes the transport of the phase field. The discrete formulation must likewise control spurious velocity components associated with such fields. 

For the class of regular potentials considered here, we prove existence and uniqueness of a global weak solution for every initial phase field $c_0\in H^1(\Gamma)$. The solution conserves mass and satisfies an energy identity. 
We prove continuous dependence on the initial data in the energy norms: the phase difference in $L^\infty(0,T;H^1(\Gamma))$ and the velocity and chemical potential differences in $L^2(0,T;H^1(\Gamma))$ are bounded by the initial phase perturbation in $H^1(\Gamma)$. The estimate also controls the pressure and the phase time derivative and allows initial data with different means. For $d=2$, we derive higher regularity results under corresponding smoothness assumptions on the surface, potential, and initial data. Related results for the full NSCH system include global strong solutions and regularity in two-dimensional domains, local strong solutions and eventual regularity in three-dimensional domains~\cite{Abels2009,GiorginiMiranvilleTemam2019,AbelsGarckeGiorgini2024}, and global strong solutions on prescribed evolving two-dimensional surfaces~\cite{AbelsGarckePoiatti2025global}.

The paper also introduces a fully discrete \emph{structure-preserving} finite element method that conserves the total phase and satisfies an energy identity for the stabilized discretization.  Finite element methods for the surface Stokes problem have been developed and analyzed in, e.g.,~\cite{OlshanskiiQuainiReuskenYushutin2018,BonitoDemlowLicht2020,JankuhnOlshanskiiReuskenZhiliakov2021,BrandnerJankuhnPraetoriusReuskenVoigt2022}. Numerical methods for surface NSCH equations were considered in~\cite{NitschkeVoigtWensch2012,PalzhanovZhiliakovQuainiOlshanskii2021,BachiniKrauseVoigt2023}, and an extension to a two-phase fluid deformable surface was studied in~\cite{BachiniKrauseNitschkeVoigt2023}. Our discretization uses the trace finite element method (TraceFEM)~\cite{OlshanskiiReuskenGrande2009,OlshanskiiReusken2017}, with an implicitly represented surface and a background mesh that need not fit the surface. We use bulk Taylor--Hood elements for velocity and pressure and continuous finite elements for the phase field and chemical potential. We recover approximate Killing fields from eigenmodes associated with small eigenvalues of a discrete surface viscous operator. For $d=2$, we also introduce an alternative approach based on restrictions of ambient rigid motions to $\Gamma$. This reduces the identification to a six-dimensional eigenvalue problem and recovers the full Killing space when all Killing fields are extrinsic.

Eigenvalue-based identification and removal of Killing fields for the surface Stokes problem was previously studied in~\cite{BonitoDemlowLicht2020}. For the coupled problem, we introduce a penalty on the approximate Killing component. The approach in~\cite{BonitoDemlowLicht2020} restores coercivity by adding a small friction term that vanishes with mesh refinement. Our stabilization instead penalizes only the Killing component. At the continuous level, it corresponds to the coercive bilinear form
\[
\big(D_\G \bu, D_\G \bv \big)_{L^2(\Gamma)} + \beta \big(\cP_E \bu, \cP_E \bv\big)_{L^2(\Gamma)},
\]
where $D_\G$ is the symmetric covariant derivative, $\beta > 0$ is the stabilization parameter, and $\cP_E$ is the projection onto Killing fields. The added term vanishes for the normalized solution $\bu\perp E$. In the discrete method, we use the approximate Killing projection and apply its orthogonal complement in both the transport and capillary terms. This removes the identified Killing component from phase transport while preserving the cancellation underlying the discrete energy law.

The remainder of the paper is organized as follows. Section~\ref{sec:governing-equations} introduces the governing equations and their conservation and dissipation properties. Section~\ref{sec:preliminaries} collects the functional setting and properties of Killing fields. In Section~\ref{sec:weak-formulation}, we establish well-posedness, continuous dependence on the initial data, and higher regularity of weak solutions. Section~\ref{sec:method} presents the spatial and temporal discretizations and their structure-preserving properties. Section~\ref{sec:numerical-experiments} discusses the identification of approximate Killing fields and reports simulations on a torus and a phase-separated liposome. Auxiliary analytical results are proved in Appendix~\ref{sec:auxiliary-results}.

\section{Governing equations} \label{sec:governing-equations}
We consider a closed, connected \(d\)-dimensional hypersurface \(\G\) embedded in \(\bbR^{d+1}\), where \(d=2,3\). We assume that \(\G\) is at least of class \(C^2\); additional smoothness is specified where needed.

Let $\dst$ be the signed distance to $\G$, positive outside. If $\G$ is of class $C^m$, $m \geq 2$, then for sufficiently small $\delta>0$, $\dst$ is of class $C^m$ in $U_\delta:=\{\bx\in\bbR^{d+1}:|\dst(\bx)|<\delta\}$~\cite[Lemma~14.16]{GilbargTrudinger2001}.
In this neighborhood, define the unit normal $\bn:=\nabla\dst$, the tangential projection $\bP:=\bI-\bn\bn^T$, and the closest-point projection $\bp(\bx):=\bx-\dst(\bx)\bn(\bx)$. The constant normal extension of a function $v$ on $\G$ is $v^e:=v\circ\bp$; vector and tensor fields are extended componentwise.

For sufficiently regular scalar and vector fields, we define $\nablaG v:=\nabla v^e$ and $\nablaG\bv:=\bP\nabla\bv^e$, respectively, and denote the symmetric part of the latter by $D_\G\bv$. The surface divergence is $\divG\bv:=\tr(\nablaG\bv)$ and acts row-wise on tensor fields. We use $\DeltaG v:=\divG\nablaG v$ and $\Delta_S\bv:=\bP\divG D_\G\bv$ for the scalar Laplace--Beltrami and surface viscous operators~\cite{JankuhnOlshanskiiReusken2018}.

The space of infinitesimal rigid tangential motions of $\Gamma$, i.e. the Killing vector fields, is
\[
E := \{\bv \in H^1(\G)^{d+1} : \bv\cdot\bn=0,~~ D_\G \bv = 0 \text{ on } \G\}.
\]
For smooth hypersurfaces, $\dim E\leq d(d+1)/2$~\cite[\S III.6, Proposition~6.5]{Sakai1996}. For $C^2$ hypersurfaces, finite-dimensionality follows from Korn's inequality and compactness; see~\cite[Theorem~4.3]{BruersLehrenfeldWardetzky2025}. Thus $E$ is closed in $\bL^2_t(\G):=\{\bv\in L^2(\G)^{d+1}:\bv\cdot\bn=0\}$, giving the $L^2(\G)$-orthogonal decomposition
\begin{equation} \label{eq:decomp}
    \bL^2_t(\G) = E \oplus E^\perp.
\end{equation}
We let $\cP_E$ be the $\bL^2(\G)$-orthogonal projection onto $E$, and $\cP_E^\perp := \operatorname{id} - \cP_E$.

\subsection{The Stokes--Cahn--Hilliard  system} \label{subsec:sch-system}
        We consider the surface Stokes--Cahn--Hilliard  system governing tangential flow and phase separation in viscous fluid membranes: given the initial phase field $c_0$ on $\Gamma$, find the  fluid velocity field $\bu$ tangential to $\Gamma$  and satisfying $\bu \perp E$, the (reduced) fluid pressure  $p$,  the phase field $c$, and  the chemical potential $\mu$ solving
        \begin{subequations} \label{eq:sch}
            \begin{align}
                - 2 \eta \Delta_S \bu + \nablaG p & = - \sigma c \nablaG \mu, \label{eq:sch-mom} \\
                \divG \bu & = 0, \label{eq:sch-div} \\
                \pa_t c + \bu \cdot \nablaG c - M \DeltaG \mu & = 0, \label{eq:sch-phase} \\
                \mu & = \varepsilon^{-1} f_0'(c) - \varepsilon \DeltaG c. \label{eq:sch-chem}
            \end{align}
        \end{subequations}
        Here $\eta>0$ is the fluid viscosity, $\sigma>0$ is the line (surface for $d=3$) tension coefficient, $M>0$ is the mobility, $\varepsilon>0$ is the interface width parameter, and $f_0$ is a suitable double-well potential. 
        
        The inertia-free surface flow equations also occur in one-phase membrane models with Helfrich bending forces~\cite{ArroyoDeSimone2009,TorresSanchezMillanArroyo2019,PorrmannVoigt2024}. It also follows as a formal limit of the standard surface H-model     \cite{NitschkeVoigtWensch2012,PalzhanovZhiliakovQuainiOlshanskii2021} when the surface flow Reynolds number is small and so the inertia terms are neglected. The latter is a realistic assumption given the estimated values of  area density, viscosity and line tension coefficients of multi-component lipid bilayers~\cite{WangPalzhanovQuainiOlshanskiiMajd2022lipid}. 
        \begin{remark}[the $\bu\perp E$ constraint]
        	The surface viscous operator has the Killing fields as its kernel.
        	In the absence of inertia and friction, the momentum equation therefore determines the velocity only up to the addition of a Killing field. We remove this indeterminacy by seeking the tangential velocity 	$L^2(\Gamma)$-orthogonal to all Killing fields. 	An alternative, considered, e.g., in \cite{PorrmannBartelsVoigt2026}, is to add the zeroth-order term 	$\gamma\bu$, with constant $\gamma>0$, to the left-hand side 	of the momentum equation. This term models dissipation due to 	friction with an ambient medium and removes the need to impose 	$\bu\perp E$ explicitly.
        	
        	A key observation we made is that the capillary force is 	$L^2(\Gamma)$-orthogonal to all Killing fields, as established 	in Proposition~\ref{prop:cap-force}. This ensures compatibility of 	the momentum equation with the kernel of the viscous operator. 	The system \eqref{eq:sch} is therefore consistent and, as shown
        	by the energy law below, dissipative without an additional 	friction term. If $\gamma\bu$ is included, testing the momentum 	equation with Killing fields shows that $\bu\perp E$ holds
        	automatically. The energy law then contains the additional 	dissipation term $\gamma\|\bu\|^2$, and the analysis of this 	paper extends with minor modifications.
        \end{remark}

We will write $(\cdot,\cdot)$ for the $L^2(\G)$ inner product of scalar, vector, or tensor fields, with Euclidean or Frobenius contraction as appropriate, and $\|v\|:=(v,v)^{1/2}$. 

\begin{proposition}[capillary force is orthogonal to Killing fields] \label{prop:cap-force} 
        Assume $c$ and $\mu$ solving \eqref{eq:sch} are sufficiently smooth. The following orthogonality property holds:
    \begin{equation}\label{eq:orth}
        {(c\nablaG\mu,\bv_E)=0}\qquad  \text{for any }~ \bv_E\in E.
    \end{equation}  
    \end{proposition}
\begin{proof}
    To check \eqref{eq:orth}, we use the divergence free property of Killing fields, $\divG \bv_E = 0$,  together with the equation for the chemical potential, $\mu = \varepsilon^{-1} f_0'(c) - \varepsilon \DeltaG c$, to see that
        \[
            {-(c\nablaG\mu,\bv_E)=(\bv_E\cdot\nablaG c,\mu)=\varepsilon^{-1}(\bv_E\cdot\nablaG c,f_0'(c))-\varepsilon(\bv_E\cdot\nablaG c,\DeltaG c).}
        \]
        Since $f_0'(c) \nablaG c = \nablaG f_0(c)$ and $\divG \bv_E = 0$, the first term above vanishes. Integrating the remaining term by parts yields
        \[
            {-(\bv_E\cdot\nablaG c,\DeltaG c)=([\nablaG\bv_E]\nablaG c,\nablaG c)+(\bv_E,[\nablaG^2 c]\nablaG c).}
        \]
        The first term above vanishes because $\nablaG \bv_E$ is skew-symmetric. For the last term, we note that $\nablaG^2 c=[\nablaG^2 c]^T$ and so $\frac12 \nablaG |\nablaG c|^2 = [\nablaG^2 c]  \nablaG c$. Hence the last term reads ${\frac12(\bv_E,\nablaG|\nablaG c|^2)}$ and thus also vanishes because $\divG \bv_E = 0$.
\end{proof}

        Smooth solutions of \eqref{eq:sch} satisfy an energy balance and conserve the total phase. 
        To derive an energy identity, we multiply the momentum equation \eqref{eq:sch-mom} by $\bu$, the phase field equation \eqref{eq:sch-phase} by $\sigma \mu$, and the chemical potential equation \eqref{eq:sch-chem} by $\sigma \pa_t c$. {Integrating over $\G$ and using} $(-\Delta_S \bu, \bu) = (D_\G \bu, D_\G \bu)$, $(\nablaG p, \bu) = -(p, \divG\bu) =0$, and $(-\DeltaG v, w) = (\nablaG v, \nablaG w)$, we deduce
        \begin{subequations} \label{eq:en-law-derivation-1}
            \begin{align}
                2 \eta \|D_\G \bu\|^2 & = - \sigma (c \nablaG \mu, \bu), \label{eq:en-law-derivation-1-mom} \\
                \sigma (\pa_t c, \mu) + \sigma (\bu \cdot \nablaG c, \mu) + \sigma M \|\nablaG \mu\|^2 & = 0, \label{eq:en-law-derivation-1-phase} \\
                \sigma (\mu, \pa_t c) & = \sigma \left[ \varepsilon^{-1} (f_0'(c), \pa_t c) + \varepsilon (\nablaG c, \nablaG \pa_t c) \right]. \label{eq:en-law-derivation-1-chem}
            \end{align}
        \end{subequations}
        Noting that $(f_0'(c), \pa_t c) = \ddt (f_0(c), 1)$ and $(\nablaG c, \nablaG \pa_t c) = (\nablaG c, \pa_t \nablaG c) = \frac12 \ddt \|\nablaG c\|^2$, the chemical potential equation \eqref{eq:en-law-derivation-1-chem} can be further rewritten as
        \begin{equation} \label{eq:cE}
            \sigma (\mu, \pa_t c) = \ddt \cE[c(t)], \qquad \cE[v] := \sigma \left( \frac{\varepsilon}{2} \|\nablaG v\|^2 + \varepsilon^{-1} (f_0(v), 1) \right).
        \end{equation}
        Combining the terms in \eqref{eq:en-law-derivation-1-mom}, \eqref{eq:en-law-derivation-1-phase} and \eqref{eq:cE} noting that $(\bu \cdot \nablaG c, \mu) = - (\bu \cdot \nablaG \mu, c)$ since $\divG \bu = 0$, and integrating from $0$ to $t$ for $t > 0$, we arrive at the following energy law:
        \begin{equation} \label{eq:en-law}
            \cE[c(t)] + \int_0^t \left[ 2 \eta \|D_\G \bu\|^2 + \sigma M \|\nablaG \mu\|^2 \right] \ds = \cE[c_0].
        \end{equation}
        Moreover, integrating the phase field equation \eqref{eq:sch-phase} over $\G$ leads to the conservation of total phase: 
        \[
        {(c(t),1)=(c_0,1).}
        \]

    
\section{Preliminaries} \label{sec:preliminaries} 
        Let $L^p(\G)$ and $W^{k,p}(\G)$ be, respectively, the Lebesgue and Sobolev spaces on $\G$; we write $H^k(\G) := W^{k,2}(\G)$ and $W^{0,p}(\G) = L^p(\G)$. For vector- and tensor-valued fields, we shall use boldface notation (e.g., $\bL^p(\G)$) and double boldface ($\dL^p(\G)$), respectively. We shall keep the same convention when $\G$ is replaced with another spatial domain.

        For $\G$ of class $C^m$, $C^m(\G)$ is dense in $W^{k,p}(\G)$ for $k \leq m$ and $p \in [1,\infty)$ \cite[Theorem 3.6]{BenavidesNochettoShakipov2026-scalar}. The definitions of differential operators {from Section~\ref{sec:governing-equations}} can thus be extended to Sobolev spaces via density.

       We use the following subspaces for $k\geq0$ and $p\in[1,\infty]$:
        \begin{equation} \label{eq:subspaces}
            \begin{split}
                C^k_\#(\G) := \{v \in C^k(\G) : (v, 1) = 0\}, \quad & \bC^k_t(\G) := \{\bv \in \bC^k(\G) : \bv \cdot \bn = 0\}, \\
                W^{k,p}_\#(\G) := \{v \in W^{k,p}(\G) : (v, 1) = 0\}, \quad & \bW^{k,p}_t(\G) := \{\bv \in \bW^{k,p}(\G) : \bv \cdot \bn = 0\}.
            \end{split}
        \end{equation}
        It is straightforward to verify that $C^m_\#(\G)$ is dense in $W^{k,p}_\#(\G)$ for $k \leq m$ {and $p\in[1,\infty)$}. For tangential vector fields we have: for $\G$ of class $C^m, m \geq 2$, integer $1 \leq k \leq m-1$, $\bC^{m-1}_t(\G)$ is dense in $\bW^{k,p}_t(\G)$ for $p \in [1, \infty)$ \cite[Lemma 3.10]{BenavidesNochettoShakipov2026-scalar}. Note the loss of regularity due to the tangentiality assumption.
        
        The following Poincar\'e inequality is valid for $\G$ of class $C^2$: for all $v \in H^1_\#(\G)$,
        \begin{equation} \label{eq:poincare}
            \|v\| \leq C_P \|\nablaG v\|.
        \end{equation}

        {To state the identities used below, we introduce the Weingarten map $\bB:=\nablaG\bn$, which is symmetric and satisfies $\bB\bn=\bzero$, and the Ricci curvature map $\bR:=\tr(\bB)\bB-\bB^2$. The next result collects the properties of Killing fields needed in the analysis.}
        \begin{proposition}[properties of Killing fields] \label{prop:Killing}
            Let $\G$ be of class $C^m, m \geq 2$. Then the following properties of Killing vector fields are valid:
            \begin{enumerate}
                \item For all $\bv \in E$, $\divG \bv = 0$.
                \item $E \subseteq \bW^{m-1,p}_t(\G)$ for any $p \in (1,\infty)$.
                \item Decomposition \eqref{eq:decomp} is stable in $\bH^{m-1}(\G)$, i.e., there exists a constant $C > 0$ such that for all $\bv \in \bH^{m-1}_t(\G)$, 
                \begin{equation} \label{eq:stab-Killing-decomp}
                    \|\cP_E \bv\|_{\bH^{m-1}(\G)} + \|\cP_E^\perp \bv\|_{\bH^{m-1}(\G)} \leq C \|\bv\|_{\bH^{m-1}(\G)}.
                \end{equation}
                \item The definition of $\cP_E$ extends to $\bL^p_t(\G), p \in [1,\infty]$ with stability constant $ C(\G, p) > 0$.
            \end{enumerate}
        \end{proposition}
        \begin{proof}
            Property (1) follows from $\divG = \tr(D_\G)$. For property (2), Lemma \ref{lem:bP-divG-nablaGT}  yields that $\bu \in \bH^1_t(\G)$ with $\divG \bu = 0$ satisfies $((\nablaG \bu)^T, \nablaG \bw) = - (\bR \bu, \bw)$ for all $\bw \in \bH^1_t(\G)$. Thus, for $\bv \in E$, the above identity and $D_\Gamma\bv=0$ {imply} $(\nablaG \bv, \nablaG \bw) = (\bR \bv, \bw)$ for all $\bw \in \bH^1_t(\G)$. This is the weak formulation of the Bochner--Laplace problem $- \Delta_B \bv := - \bP \divG \nablaG \bv = \bR \bv$, from which higher regularity results for $\bv \in E$ follow thanks to \cite[Theorems 3.6, 3.7]{BenavidesNochettoShakipov2025-vector}.

            Property (3) follows from the fact that $E$ is finite-dimensional and the regularity of Killing fields. For property (4), we let $\bv \in \bL^2_t(\G)$ and write $\cP_E \bv = \sum_{i=1}^{\dim E} (\bv, \be_i) \be_i$, where $\{\be_i\}_{i=1}^{\dim E}$ is an $\bL^2(\G)$-orthonormal basis of $E$. We then see that this definition extends to $\bL^1_t(\G)$ since by (2) and a Sobolev embedding $\be_i \in \bC^0_t(\G)$ and so  we infer $\|\cP_E \bv\|_{\bL^p(\G)} \leq C \|\bv\|_{\bL^p(\G)}$, with the finite $C = \sum_{i=1}^{\dim E} \|\be_i\|_{\bL^p(\G)} \|\be_i\|_{\bL^{p'}(\G)}$. This finishes the proof.
        \end{proof}
        We next define
        \begin{equation} \label{eq:bHk-tperp}
            \bW^{k, p}_{t, \perp}(\G) := \{\bv \in \bW^{k, p}_t(\G) : \cP_E \bv = \bzero\}, \quad k \geq 0, \quad p \in [1, \infty].
        \end{equation}
        {We write $A\lesssim B$ if $A\leq CB$ with a constant $C$ independent of the approximation indices or mesh parameters under consideration; other dependencies are specified when needed.} The following Korn inequality is also valid for $\G$ of class $C^2$:
        \begin{equation} \label{eq:korn}
            \|\bv\|_{\bH^1(\G)} \leq C_K \|D_\G \bv\|, \quad \forall \bv \in \bH^1_{t, \perp}(\G).
        \end{equation}
        Here $C_K = C_K(\G) > 0$. Inequality \eqref{eq:korn} is proved in \cite[Lemma 4.1]{JankuhnOlshanskiiReusken2018} for the complement of $E$ in $\bH^1_t(\G)$ taken orthogonal with respect to $\bH^1(\G)$ rather than $\bL^2(\G)$. This is sufficient for our purposes: since $\dim E < \infty$ and $E \subseteq \bH^1_t(\G)$, there is $C_E = C_E(\G) > 0$ with $\|\be\|_{\bH^1(\G)} \le C_E \|\be\|$ for all $\be \in E$, and writing $\bv \in \bH^1_{t,\perp}(\G)$ as $\bv = \bv_0 + \bw$ with $\bv_0$ in the $\bH^1$-orthogonal complement and $\bw \in E$ gives $\bw = -\cP_E \bv_0$ and $D_\G \bv = D_\G \bv_0$, whence $\|\bv\|_{\bH^1(\G)} \le (1+C_E)\|\bv_0\|_{\bH^1(\G)} \lesssim \|D_\G \bv\|$.

    
    \section{Weak formulation} \label{sec:weak-formulation}
    In this section, we introduce  a notion of weak solution to the surface Stokes--Cahn--Hilliard problem and establish its well-posedness. For $d=2$, we also derive a higher regularity result for {weak} solutions.
  \smallskip
  
        We introduce Dirichlet forms 
        \[ 
            \ba(\bu, \bv) := (D_\G \bu, D_\G \bv)\quad \text{and}\quad a(u, v) := (\nablaG u, \nablaG v)
        \] 
        for $\bu, \bv \in \bH^1(\G)$ and  $u, v \in H^1(\G)$, respectively. For $w \in H^1(\G)$, let $\mean{w} := |\G|^{-1}(w,1)$ and $\zm{w} := w - \mean{w}$. We define the  \emph{centered} trilinear form $b : \bH^1(\G) \times H^1(\G) \times H^1(\G) \to \bbR$ as 
        \begin{equation} \label{eq:b-form}
            {b(\bu,v,w):=(\bu\cdot\nablaG v,\zm{w}),}
        \end{equation}
        for all $(\bu, v, w) \in \bH^1(\G) \times H^1(\G) \times H^1(\G)$. The form $b$ is continuous due to the $H^1(\G) \hookrightarrow L^4(\G)$ Sobolev embedding and $H^1(\G)$-stability of $\zm{(\cdot)}$. We note that $b(\bu,v,1)=0$ for all $(\bu,v) \in \bH^1(\G) \times H^1(\G)$. Furthermore, if $\bu \cdot \bn = 0, \divG \bu=0$, then $b$ agrees with the usual (non-centered) convection form and is skew-symmetric with respect to its second and third arguments. In the momentum equation, using the centered form $b(\bw, c, \mu)$ instead of ${-(\bw\cdot\nablaG\mu,c)}$ corresponding to the strong form \eqref{eq:sch-mom} amounts only to the pressure shift $p\mapsto p+\sigma\zm{(c\mu-\mean{\mu}c)}$, since
        \[
            \begin{split}
                b(\bw, c, \mu) & = (\bw \cdot \nablaG c, \zm{\mu}) = (\bw \cdot \nablaG c, \mu) - \mean{\mu} (\bw, \nablaG c) \\
                & = - (\bw \cdot \nablaG \mu, c) - (\divG \bw, c \mu) - (\bw, \nablaG [\mean{\mu}  {c}]) \\
                & = - (\bw \cdot \nablaG \mu, c) + (\bw, \nablaG [ {(c \mu)} - \mean{\mu}  {c}]).
            \end{split}
        \]
        This consistent modification of the momentum equation builds a direct cancellation of the capillary force term $b(\bw, c, \mu)$ and the convective term $b(\bu, c, v)$ in the phase field equation when testing with $\bw = \bu, v=\mu$, thereby circumventing condition $\divG \bu = 0$. This is critical for developing the structure-preserving numerical method where the divergence constraint is hard to enforce pointwise.
        
        For the rest of the paper, we make the following assumption on the double-well potential $f_0$.
        \begin{assump} \label{assump:double-well}
            The double-well potential $f_0 \in C^{2,1}_{loc}(\bbR)$ satisfies $f_0 \geq 0$ and the following growth bound for some $L > 0$:
            \begin{equation} \label{eq:double-well}
                |f_0'''(s)| \leq L (1 + |s|) \quad \text{for a.e. }s\in\bbR. \tag{DW}
            \end{equation}
        \end{assump}
        Assumption \ref{assump:double-well} is sufficient to prove the existence of weak solutions and continuous dependence results. For higher regularity statements we shall naturally require additional smoothness on $f_0$. Condition \eqref{eq:double-well} is not very restrictive; most regular potentials such as the Ginzburg--Landau potential $f_{\mathrm{quart}}(s) = \frac{(s^2-1)^2}{4}$  satisfy it.
        
        Let $T > 0$ be the terminal time and $I := (0, T)$ be a time interval. The weak formulation of the Stokes--Cahn--Hilliard problem studied in this paper reads: Given initial datum $c_0 \in H^1(\G)$, find 
        \begin{equation}\label{eq:trialSpace}
            \begin{alignedat}{2}
                &\bu\in L^2(I;\bH^1_{t,\perp}(\G)),\qquad &&c\in L^\infty(I;H^1(\G))\cap H^1(I;H^1(\G)'),\\
                &p\in L^2(I;L^2_\#(\G)), &&\mu\in L^2(I;H^1(\G)),
            \end{alignedat}
        \end{equation}
        such that for a.e. $t \in I$ and all $(\bw, q, v, g) \in \bH^1_{t, \perp}(\G) \times L^2_\#(\G) \times H^1(\G) \times H^1(\G)$,
        \begin{subequations} \label{eq:weak}
            \begin{align}
                2 \eta \, \ba(\bu, \bw) - (p, \divG \bw) & = \sigma b(\bw, c, \mu), \label{eq:weak-mom} \\
                (q, \divG \bu) & = 0, \label{eq:weak-div} \\
                \langle \pa_t c, v \rangle_{H^1(\G)' \times H^1(\G)} + b(\bu, c, v) + M a(\mu, v) & = 0, \label{eq:weak-phase} \\
                (\mu, g) - \varepsilon^{-1} (f_0'(c), g) - \varepsilon a(c, g) & = 0, \label{eq:weak-chem}
            \end{align}
        \end{subequations}
        {with $c(0)=c_0$.}

    \subsection{Well-posedness} \label{subsec:well-pos}
        In this subsection, we establish {existence and uniqueness} of weak solutions to \eqref{eq:weak} 
        and {their} continuous dependence on the initial data. 
        
        \begin{theorem}[existence of weak solutions] \label{thm:existence}
            Let $\G$ be of class $C^2$ and $d \in \{2, 3\}$. For every $c_0\in H^1(\G)$, the weak formulation \eqref{eq:weak} has a global solution $(\bu, p, c, \mu)$ satisfying \eqref{eq:weak}. The weak solution satisfies {the} energy \emph{inequality}
                    \begin{equation*} 
            \cE[c(t)] + \int_0^t \left[ 2 \eta \|D_\G \bu\|^2 + \sigma M \|\nablaG \mu\|^2 \right] \ds \le \cE[c_0]\qquad \text{for a.e. } t \in I.
        \end{equation*}
     Moreover, $(c(t), 1) = (c_0, 1)$ for all $t \in [0,T]$.
        \end{theorem}
        \begin{proof}
            The existence proof is based on the Faedo--Galerkin method; see, e.g., \cite{Elliott1989, CaetanoElliott2021}. We proceed in several steps. \smallskip

            \noindent \textbf{Faedo--Galerkin scheme}. Let $\{\varphi_k\}_{k=1}^\infty$ be an orthonormal basis of $L^2(\G)$ consisting of eigenfunctions of $\operatorname{id} - \Delta_\G$ ordered so that $\varphi_1$ is constant, and let $\{\bvarphi_k\}_{k=1}^\infty$ be a basis of 
            $$\bV := \{\bv \in \bH^1_{t, \perp}(\G) : \divG \bv = 0\}.$$ 
            For a fixed integer $N \geq 1$, we let $V_N := \operatorname{span} \{\varphi_k\}_{k=1}^N, \bV_N := \operatorname{span} \{\bvarphi_k\}_{k=1}^N$, and seek approximate solutions of the form
        	\[
            	c_N(t)=\sum_{i=1}^N a_i(t)\varphi_i, \qquad \mu_N(t)=\sum_{i=1}^N b_i(t)\varphi_i, \qquad \bu_N(t)=\sum_{i=1}^N d_i(t) \bvarphi_i.
        	\]
        	Let $\Pi_N : L^2(\G) \to V_N$ be the $L^2(\G)$-orthogonal projection. Since $\{\varphi_k\}_{k=1}^\infty$ are Laplace--Beltrami eigenfunctions, $\Pi_N$ is $H^1$-orthogonal and hence stable: $\|\Pi_N v\|_{H^1(\G)} \le \|v\|_{H^1(\G)}$ for all $v\in H^1(\G)$.
        	
            Consider the following semidiscrete system: Given {$c_N(0)=\Pi_N c_0$}, find $\bu_N \in L^2(I; \bV_N), c_N \in H^1(I; V_N), \mu_N \in L^2(I; V_N)$ such that for a.e. $t \in I$ and all $(\bw_N, v_N, g_N) \in \bV_N \times V_N \times V_N$,
            \begin{subequations} \label{eq:Faedo-Galerkin}
                \begin{align}
                    2 \eta \, \ba(\bu_N, \bw_N) & = \sigma b(\bw_N, c_N, \mu_N), \label{eq:Faedo-Galerkin-mom} \\
                    (\pa_t c_N, v_N) + b(\bu_N, c_N, v_N) + M a(\mu_N, v_N) & = 0, \label{eq:Faedo-Galerkin-phase} \\
                    (\mu_N, g_N) & = \varepsilon^{-1} (f_0'(c_N), g_N) + \varepsilon a(c_N, g_N). \label{eq:Faedo-Galerkin-chem}
                \end{align}
            \end{subequations}   
            The momentum and chemical potential equations in~\eqref{eq:Faedo-Galerkin} determine $\bu_N$ and $\mu_N$ uniquely from $c_N$, using Korn's inequality~\eqref{eq:korn} for the momentum equation. Substitution into the phase equation yields an ODE system with a locally Lipschitz right-hand side, which gives local existence. The a priori bounds below extend the solution to $[0,T]$.\smallskip

            \noindent \textbf{A priori bounds}. The energy estimate \eqref{eq:en-law} applies verbatim to the Galerkin approximations after testing \eqref{eq:Faedo-Galerkin-mom} with $\bw_N = \bu_N \in \bV_N$, \eqref{eq:Faedo-Galerkin-phase} with $v_N = \sigma \mu_N \in V_N$, and \eqref{eq:Faedo-Galerkin-chem} with $g_N = \pa_t c_N \in V_N$. This establishes an a priori estimate for a.e. $t \in I$ (after extending the scheme to $[0,T]$)
            \begin{equation} \label{eq:en-law-Galerkin}
                \cE[c_N(t)] + \int_0^t \left[ 2 \eta \|D_\G \bu_N\|^2 + \sigma M \|\nablaG \mu_N\|^2 \right] \ds \leq \cE[\Pi_N c_0].
            \end{equation}
            Since $\cE[\Pi_N c_0] \lesssim \|\nablaG \Pi_N c_0\|^2 + \|f_0(\Pi_N c_0)\|_{L^1(\G)}$ and $f_0$ is at most quartic by \eqref{eq:double-well}, we infer by Sobolev embedding that $\cE[\Pi_N c_0] \lesssim \|\Pi_N c_0\|_{H^1(\G)}^2 + 1 + \|\Pi_N c_0\|_{L^4(\G)}^4 \lesssim 1 + \|\Pi_N c_0\|_{H^1(\G)}^4 \lesssim 1 + \|c_0\|_{H^1(\G)}^4$ since $\Pi_N$ is $H^1(\G)$-orthogonal {and hence} stable. Therefore, bound \eqref{eq:en-law-Galerkin} is uniform in $N$.
            
            Taking $v_N = 1 \in V_N$ in \eqref{eq:Faedo-Galerkin-phase} implies $(c_N(t), 1) = (c_0, 1)$, from which {Poincar\'e's inequality} combined with \eqref{eq:en-law-Galerkin} yields that $c_N$ is bounded in $L^\infty(I; H^1(\G))$. Similarly, \eqref{eq:en-law-Galerkin} combined with Korn's inequality \eqref{eq:korn} yields that $\bu_N$ is bounded in $L^2(I; \bH^1(\G))$. Taking $g_N = 1 \in V_N$ in \eqref{eq:Faedo-Galerkin-chem}, using $|f_0'(s)| \leq C(1 + |s|^3)$ for all $s \in \bbR$ (implied by \eqref{eq:double-well}) followed by Sobolev embedding gives $|(\mu_N, 1)| = \varepsilon^{-1} |(f_0'(c_N), 1)| \leq C \varepsilon^{-1} (1 + \|c_N\|_{H^1(\G)}^3)$, so the mean of $\mu_N$ is bounded uniformly in $N$, hence $\mu_N $ is bounded in $L^2(I; H^1(\G))$ due to \eqref{eq:en-law-Galerkin}. Next, from \eqref{eq:Faedo-Galerkin-phase}, for every $v_N \in V_N$,
        	\[
            	(\pa_t c_N, v_N) = - b(\bu_N, c_N, v_N) - (M\nablaG\mu_N,\nablaG v_N).
        	\]
        	Note that for an arbitrary $v\in H^1(\G)$, we have $(\pa_t c_N, v) = (\pa_t c_N, \Pi_N v)$. Using continuity of $b$ and $H^1(\G)$-stability of $\Pi_N$, we obtain
        	\[
                |(\pa_t c_N, v)| \le C (\|c_N\|_{H^1(\G)} \|\bu_N\|_{\bH^1(\G)} + \|\nablaG\mu_N\|_{\bL^2(\G)}) \|v\|_{H^1(\G)}.
        	\]
        	Hence $\pa_t c_N$ is bounded in $L^2(I; H^1(\G)')$ due to uniform bounds on $\|c_N\|_{L^\infty(I; H^1(\G))}$, $\|\bu_N\|_{L^2(I; \bH^1(\G))}$, and $\|\nablaG \mu_N\|_{L^2(I; \bL^2(\G))}$. \smallskip
            
            \noindent \textbf{Compactness argument}. By the Aubin--Lions--Simon compactness lemma \cite[p.~85,~Cor.~4]{Simon1987}, there exists a subsequence, still denoted by $\{c_N,\mu_N,\bu_N\}$, such that
        	\[
                c_N \to c \quad \text{strongly in } C^0([0,T]; L^q(\G)), \quad q \in [1, 6),
        	\]
        	and also almost everywhere on $\G \times I$, while the Banach--Alaoglu theorem yields
            \[
            \begin{split}
                c_N \rightharpoonup^\ast c \quad\text{ in } L^\infty(I; H^1(\G)), & \qquad \pa_t c_N \rightharpoonup \pa_t c \quad \text{ in } L^2(I; H^1(\G)'), \\
            	\mu_N \rightharpoonup \mu \quad\text{in } L^2(I; H^1(\G)), & \qquad \bu_N \rightharpoonup \bu \quad \text{ in } L^2(I; \bH^1_{t, \perp}(\G)).
            \end{split}
            \]
            \textbf{Passage to the limit}. Weak convergence suffices to pass to the limit in the linear terms of~\eqref{eq:Faedo-Galerkin}. Since $|f_0'(s)| \leq C(1 + |s|^3)$ for all $s \in \bbR$, the strong convergence of $c_N$ (pick $q = 4$) implies
        	\begin{equation} \label{eq:f0'-strong-conv}
        	    f_0'(c_N)\to f_0'(c) \quad\text{strongly in } C^0([0,T]; L^{4/3}(\G)),
        	\end{equation}
            yielding $(f_0'(c_N), g) \xrightarrow{N \to \infty} (f_0'(c), g)$ for all $g \in L^2(I; H^1(\G))$. {It remains to pass to the limit in the two convective terms.} 
            \begin{enumerate}
                \item \textit{Term $(\bw \cdot \nablaG c, \mu)$}. Let $\bw \in L^2(I; \bV)$. We have
                \begin{equation} \label{eq:exist-first-conv-term}
                    \begin{split}
                        \int_I |(\bw \cdot \nablaG c_N, \mu_N)| \dt & \leq \int_I \|\nablaG c_N\|_{\bL^2(\G)} \|\mu_N\|_{L^4(\G)} \|\bw\|_{\bL^4(\G)} \dt \\
                        & \leq C \|c_N\|_{L^\infty(I; H^1(\G))} \|\mu_N\|_{L^2(I; H^1(\G))} \|\bw\|_{L^2(I; \bH^1(\G))} \\
                        & \leq C \|\bw\|_{L^2(I; \bH^1(\G))},
                    \end{split}
                \end{equation}
                where we used $\divG \bw = 0$ to replace $\zm{\mu_N}$ with $\mu_N$ and a priori bounds. Therefore, $\mu_N \nablaG c_N$ is bounded uniformly in $L^2(I; \bV')$. Using $\divG \bw = 0$ again, we write
                \begin{equation} \label{eq:ident-conv-term}
                    \begin{split}
                        \int_I (\bw \cdot \nablaG c_N, \mu_N) - (\bw \cdot \nablaG c, \mu) \dt & = \int_I (\bw \cdot \nablaG [c_N - c], \mu_N) + (\bw \cdot \nablaG c, \mu_N - \mu) \dt \\
                        & = \int_I (\bw \cdot \nablaG \mu_N, c - c_N) + (\bw \cdot \nablaG c, \mu_N - \mu) \dt.
                    \end{split}
                \end{equation}
                The first term on the right-hand side in \eqref{eq:ident-conv-term} is controlled using the a priori bound on $\nablaG \mu_N$, and strong convergence for $c_N$ as follows:
                \[
                    \begin{split}
                        \int_I |(\bw \cdot \nablaG \mu_N, c - c_N)| \dt & \leq \int_I \|c_N - c\|_{L^4(\G)} \|\nablaG \mu_N\| \|\bw\|_{\bL^4(\G)} \dt \\
                        & \leq \|c_N - c\|_{C^0([0, T]; L^4(\G))} \|\nablaG \mu_N\|_{L^2(I; \bL^2(\G))} \|\bw\|_{L^2(I; \bL^4(\G))} \\
                        & \quad \xrightarrow{N \to \infty} 0.
                    \end{split}
                \]
                The second term on the right-hand side in \eqref{eq:ident-conv-term} is controlled using weak convergence since $\mu_N \rightharpoonup \mu \text{ in } L^2(I; L^4(\G))$ and $\|\bw \cdot \nablaG c\|_{L^2(I; L^{4/3}(\G))} \leq \|c\|_{L^\infty(I; H^1(\G))} \|\bw\|_{L^2(I; \bL^4(\G))}$, which is finite due to Sobolev embedding. Thus, for all $\bw \in L^2(I; \bV)$,
                \[
                    \int_I (\bw \cdot \nablaG c_N, \mu_N) \dt \xrightarrow{N \to \infty} \int_I (\bw \cdot \nablaG c, \mu) \dt.
                \]
                \item \textit{Term $(\bu \cdot \nablaG c, v)$}. First, note that $\bu_N \cdot \nablaG c_N \in L^2(I; H^1(\G)')$ with a uniform in $N$ bound by an estimate similar to \eqref{eq:exist-first-conv-term}. Let $v \in L^2(I; H^1(\G))$. Using $\divG \bu_N = 0$, we write
                \[
                    \begin{split}
                        \int_I (\bu_N \cdot \nablaG c_N, v) - (\bu \cdot \nablaG c, v) \dt & = \int_I (\bu_N \cdot \nablaG [c_N - c], v) + ([\bu_N - \bu] \cdot \nablaG c, v) \dt \\
                        & = \int_I (\bu_N \cdot \nablaG v, c - c_N) + ([\bu_N - \bu] \cdot \nablaG c, v) \dt.
                    \end{split}
                \]
                The second term on the right-hand side is handled similarly to the second term in \eqref{eq:ident-conv-term}. We now bound the first term above using {the} a priori bound for $\bu_N$ and strong convergence for $c_N$:
                \[
                    \begin{split}
                        \int_I |(\bu_N \cdot \nablaG v, c - c_N)| \dt & \leq \int_I \|\bu_N\|_{\bL^4(\G)} \|\nablaG v\| \|c - c_N\|_{L^4(\G)} \dt \\
                        & \leq C \|c - c_N\|_{C^0([0, T]; L^4(\G))} \|\bu_N\|_{L^2(I; \bH^1(\G))} \|v\|_{L^2(I; H^1(\G))} \\
                        & \quad \xrightarrow{N \to \infty} 0.
                    \end{split}
                \]
            \end{enumerate}
            Let $\wt N \geq 1$ be an integer and $N \geq \wt N$.
        	Passing to the limit first as $N \to \infty$, then as $\wt N \to \infty$ and using the density of $\cup_{\wt N \geq 1} V_{\wt N}$ in $H^1(\G)$ and of $\cup_{\wt N \geq 1} \bV_{\wt N}$ in $\bV$ in the Galerkin system yields a weak solution $(\bu, c, \mu)$ of \eqref{eq:weak} with the momentum equation \eqref{eq:weak-mom} restricted to the $\divG$-free subspace. Since $\bV$ is a closed subspace, hence weakly closed, $(\divG \bu, q) = 0$ for all $q \in L^2_\#(\G)$ holds in the limit. Since $c_N \to c$ in $C^0([0,T]; L^2(\G))$ and $\Pi_N c_0 \to c_0$ as $N \to \infty$, we verify that $c(0) = c_0$. \smallskip

            \noindent \textbf{Recovery of the pressure}. The limit functions $(\bu, c, \mu)$ satisfy for a.e. $t \in I$ and all $\bw \in \bV$,
            \begin{equation} \label{eq:momentum-on-kernels}
                2 \eta \, \ba(\bu, \bw) = \sigma b(\bw, c, \mu).
            \end{equation}
            Let $\bell \in \bH^1_{t, \perp}(\G)'$ be defined as
            \[
            \langle \bell, \bv \rangle_{\bH^1_{t, \perp}(\G)' \times \bH^1_{t, \perp}(\G)} := 2 \eta \, \ba(\bu, \bv) - \sigma b(\bv, c, \mu).
            \]
            Due to \eqref{eq:momentum-on-kernels}, we have $\langle \bell, \bw \rangle_{\bH^1_{t, \perp}(\G)' \times \bH^1_{t, \perp}(\G)} = 0$ for all $\bw \in \bV$. From the inf-sup condition in~\cite[Lemma 4.2]{JankuhnOlshanskiiReusken2018} and the fact that $\bv \mapsto \cP_E^\perp \bv$ leaves $\divG \bv$ {unchanged and is bounded} in $\bH^1(\G)$ by Proposition \ref{prop:Killing}, we infer that $\divG : \bH^1_{t, \perp}(\G) \to L^2_\#(\G)$ is surjective. We thus deduce that there exists a unique $p \in L^2_\#(\G)$ such that $\langle \bell, \bv \rangle_{\bH^1_{t, \perp}(\G)' \times \bH^1_{t, \perp}(\G)} = (p, \divG \bv)$ for all $\bv \in \bH^1_{t, \perp}(\G)$, and the following a priori estimate is valid:
            \[
                \begin{split}
                    \|p\|_{L^2(I; L^2(\G))} & \leq C (\|D_\G \bu\|_{L^2(I; \dL^2(\G))} + \|(\nablaG c) \zm{\mu}\|_{L^2(I; \bL^{4/3}(\G))}) \\
                    & \leq C (\|D_\G \bu\|_{L^2(I; \dL^2(\G))} + \|\nablaG c\|_{L^\infty(I; \bL^2(\G))} \|\mu\|_{L^2(I; H^1(\G))}),
                \end{split}
            \]
            where $C > 0$ depends on $\eta, \sigma$ and the inf-sup constant. \smallskip
            
            \noindent \textbf{Energy inequality and mass conservation}. 
            We pass to the limit in the Galerkin energy estimate using weak lower semicontinuity for the quadratic terms and Fatou's lemma for $f_0(c_N)\geq0$. On the right-hand side, $\Pi_N c_0\to c_0$ in $H^1(\G)$ and the quartic growth of $f_0$ imply $\cE[\Pi_N c_0]\to\cE[c_0]$. This is the standard Galerkin compactness argument; see also~\cite[Section~4.1]{CaetanoElliott2021} for the surface Cahn--Hilliard equation. It yields the energy inequality. Mass conservation follows by taking $v = 1 \in L^2(I; H^1(\G))$ in \eqref{eq:weak-phase}. This finishes the proof.
        \end{proof}
        We next derive a continuous dependence result in the norms consistent with the energy law \eqref{eq:en-law}. We start with the following lemma that establishes higher regularity of $c$ and $f_0'(c)$. This smoothness pickup follows from the chemical potential equation $\mu = \varepsilon^{-1} f_0'(c) - \varepsilon \DeltaG c$ and elliptic regularity for $\G$ of class $C^2$.

        \begin{lemma}[regularity of a weak solution] \label{lem:reg-weak}
            Let $\G$ be of class $C^2$ and $d \in \{2, 3\}$, and suppose that $(c, \mu) \in L^\infty(I; H^1(\G)) \times L^2(I; H^1(\G))$ satisfy the chemical potential equation \eqref{eq:weak-chem}. Then
            \[
                c \in L^2(I; W^{2,6}(\G)) \cap L^4(I; L^\infty(\G)), \qquad \DeltaG c \in L^2(I; H^1(\G)), \qquad f_0'(c) \in L^2(I; H^1(\G)).
            \]
        \end{lemma}
        \begin{proof}
            We proceed in several steps. \smallskip
            
            \noindent \textbf{Step 1: $c \in L^2(I; H^2(\G))$}. The cubic growth of $f_0'$ and the embedding $H^1(\G)\hookrightarrow L^6(\G)$ imply $f_0'(c)\in L^\infty(I;L^2(\G))$. Thus the right-hand side of $-\varepsilon\DeltaG c=\mu-\varepsilon^{-1}f_0'(c)$ belongs to $L^2(I;L^2(\G))$. Elliptic regularity on $C^2$ surfaces~\cite[Theorem~1.2]{BenavidesNochettoShakipov2026-scalar} gives the assertion.
            \smallskip

            \noindent \textbf{Step 2: $c \in L^4(I; L^\infty(\G))$}. We need  \textit{Agmon's inequality}: for all $v \in H^2(\G)$,
            \begin{equation} \label{eq:agmon}
                \|v\|_{L^\infty(\G)}^2 \leq C(\G) \|v\|_{H^1(\G)} \|v\|_{H^2(\G)};
            \end{equation}
            which can be verified for $\G$ of class $C^2$ by localizing and rewriting the quantities appearing in the inequality in flat domains, then using the classical Agmon's inequality in $\bbR^d$. Then,
            \[
                \|c\|_{L^\infty(\G)}^4 \leq C(\G)^2 \|c\|_{H^1(\G)}^2 \|c\|_{H^2(\G)}^2.
            \]
            Integrating in time and using $c \in L^\infty(I; H^1(\G)) \cap L^2(I; H^2(\G))$ proves the result. \smallskip

            \noindent \textbf{Step 3: $f_0'(c) \in L^2(I; H^1(\G))$}. Next we compute $\nablaG f_0'(c) = f_0''(c) \nablaG c$ and estimate using \eqref{eq:double-well} to deduce from H\"older and Sobolev embeddings that
            \[
                \begin{split}
                    \|\nablaG f_0'(c)\| & = \|f_0''(c) \nablaG c\| \leq \|f_0''(c)\|_{L^3(\G)} \|\nablaG c\|_{L^6(\G)} \leq C (1 + \|c\|_{L^6(\G)}^2) \|c\|_{H^2(\G)} \\
                    & \leq C (1 + \|c\|_{H^1(\G)}^2) \|c\|_{H^2(\G)} \leq C (1 + \|c\|_{L^\infty(I; H^1(\G))}^2) \|c\|_{H^2(\G)}.
                \end{split}
            \]
            Integrating in time and using $c \in L^\infty(I; H^1(\G)) \cap L^2(I; H^2(\G))$ proves that $f_0'(c) \in L^2(I; H^1(\G))$. \smallskip

            \noindent \textbf{Step 4: $\DeltaG c \in L^2(I; H^1(\G))$ and $c \in L^2(I; W^{2,6}(\G))$}. Since $\mu, f_0'(c) \in L^2(I; H^1(\G))$, the chemical potential equation \eqref{eq:weak-chem} readily implies $\DeltaG c \in L^2(I; H^1(\G))$. Furthermore, Sobolev embedding ensures that $\mu, f_0'(c) \in L^2(I; L^6(\G))$ and then elliptic regularity once again upgrades $c$ to $L^2(I; W^{2,6}(\G))$. This finishes the proof.
        \end{proof}
        We next need the following technical lemma that controls the perturbation error of the nonlinear term $f_0'(c_1) - f_0'(c_2)$.
        \begin{lemma}[control of the nonlinearity] \label{lem:control-of-nonlinearity}
            Let $\G$ be of class $C^2$ and $d \in \{2,3\}$. Let $\xi_1, \xi_2 \in L^\infty(I; H^1(\G)) \cap L^4(I; L^\infty(\G)) \cap L^2(I; W^{1,6}(\G))$, $\oxi := \xi_1 - \xi_2$, and $\ophi := f_0'(\xi_1) - f_0'(\xi_2)$. Then the following estimates are valid for a.e. $t \in I$:
            \[
                \begin{split}
                     {\|\ophi\|^2} & \leq C_{\phi, 0} {\|\oxi\|_{H^1(\G)}^2}, \quad \int_0^t \|\ophi\|_{H^1(\G)}^2 \ds \leq \int_0^t C_{\phi, 1}(s) \|\oxi\|_{H^1(\G)}^2 \ds,
                \end{split}
            \]
            where $C_{\phi, 0} > 0$ depends on $\G, f_0$, and the $L^\infty(I; H^1(\G))$-norms of $\xi_i$, while 
            \begin{equation} \label{eq:Cphi1}
                C_{\phi, 1}(t) = C' \left[1 + \|\xi_2(t)\|_{L^\infty(\G)}^4 + (1 + \|\xi_1\|_{L^\infty(I; H^1(\G))}^2 + \|\xi_2\|_{L^\infty(I; H^1(\G))}^2) \|\xi_1(t)\|_{W^{1,6}(\G)}^2 \right]
            \end{equation}
            belongs to $L^1(I)$ for some $C' > 0$ depending only on $\G, f_0$.
        \end{lemma}
        \begin{proof}
            Using \eqref{eq:double-well}, we prove the first estimate: 
            \[
            \begin{split}
                \|\ophi\|^2 = \|f_0'(\xi_1) - f_0'(\xi_2)\|^2 \leq C (1 + \|\xi_1\|_{L^6(\G)}^4 + \|\xi_2\|_{L^6(\G)}^4) \|\oxi\|_{L^6(\G)}^2 \leq C_\xi \|\oxi\|_{H^1(\G)}^2,
            \end{split}
            \]
            where $C_\xi := C (1 + \|\xi_1\|_{L^6(\G)}^4 + \|\xi_2\|_{L^6(\G)}^4) \in L^\infty(I)$ since each $\xi_1, \xi_2 \in L^\infty(I; H^1(\G))$.  We now turn to $\nablaG \ophi$:
            \begin{equation} \label{eq:ophi-h1-bound}
                \begin{split}
                    \|\nablaG \ophi\| & = \|f_0''(\xi_1) \nablaG \xi_1 - f_0''(\xi_2) \nablaG \xi_2\| \\
                    & \leq \|[f_0''(\xi_1) - f_0''(\xi_2)] \nablaG \xi_1\| + \|f_0''(\xi_2) \nablaG \oxi\| \\
                    & \leq \|f_0''(\xi_1) - f_0''(\xi_2)\|_{L^3(\G)} \|\xi_1\|_{W^{1,6}(\G)} + \|f_0''(\xi_2)\|_{L^\infty(\G)} \|\oxi\|_{H^1(\G)}.
                \end{split}
            \end{equation}
            Next, we use \eqref{eq:double-well}, H\"older and Sobolev embeddings as follows:
            \[
                \begin{split}
                    \|\nablaG \ophi\| & \lesssim (1 + \|\xi_1\|_{L^6(\G)} + \|\xi_2\|_{L^6(\G)}) \|\oxi\|_{L^6(\G)} \|\xi_1\|_{W^{1,6}(\G)} + (1 + \|\xi_2\|_{L^\infty(\G)}^2) \|\oxi\|_{H^1(\G)} \\
                    & \lesssim (1 + \|\xi_1\|_{H^1(\G)} + \|\xi_2\|_{H^1(\G)}) \|\oxi\|_{H^1(\G)} \|\xi_1\|_{W^{1,6}(\G)} + (1 + \|\xi_2\|_{L^\infty(\G)}^2) \|\oxi\|_{H^1(\G)},
                \end{split}
            \]
            where the hidden constant depends only on $\G$ and $f_0$. Squaring and then integrating from $0$ to $t$ and using the asserted regularity of $\xi_1, \xi_2$ finishes the proof.
        \end{proof}
        \begin{theorem}[continuous dependence in energy norms] \label{thm:cont-dep}
            Let $\G$ be of class $C^2$, $d \in \{2,3\}$. Let $(\bu_i, p_i, c_i, \mu_i)$ be two weak solutions of \eqref{eq:weak} corresponding to initial data $c_{0,i}, i \in \{1,2\}$. Then 
            \[
                \begin{split}
                    \|\bu_1 - \bu_2\|_{L^2(I; \bH^1(\G))} & + \|p_1 - p_2\|_{L^2(I; L^2(\G))} + \|c_1 - c_2\|_{L^\infty(I; H^1(\G))} \\
                    & + \|\pa_t (c_1 - c_2)\|_{L^2(I; H^1(\G)')} + \|\mu_1 - \mu_2\|_{L^2(I; H^1(\G))} \leq C_T \|c_{0,1} - c_{0,2}\|_{H^1(\G)},
                \end{split}
            \]
            where $C_T > 0$ depends on $f_0, T, \G$, material parameters, and the two solutions.
        \end{theorem}
        \begin{proof}
           Define the differences:
            \[
                \obu := \bu_1 - \bu_2, \quad \op := p_1 - p_2, \quad \oc := c_1 - c_2, \quad \omu := \mu_1 - \mu_2, \quad \ophi := f_0'(c_1) - f_0'(c_2).
            \]
            Testing \eqref{eq:weak-phase} with $v = 1 \in H^1(\G)$ for $c_1, c_2$ yields $(\oc(t), 1) = (c_1(t), 1) - (c_2(t), 1) = (c_{0,1}, 1) - (c_{0,2}, 1)$, whence $\ddt(\oc(t), 1) = 0$. We define
            \[
                \wh{c}_0 := |\G|^{-1} (c_{0,1} - c_{0,2}, 1), \qquad \oz := \oc - \wh{c}_0,
            \]
            and note that $(\oz(t), 1) = 0$ for all $t \in [0, T]$.
            Subtracting the two weak formulations \eqref{eq:weak} gives for a.e. $t \in I$,
            \begin{subequations} \label{eq:cont-err-eq}
                \begin{align}
                	2 \eta \, \ba(\obu, \bw) - (\op, \divG \bw) & = \sigma b(\bw, c_1, \mu_1) - \sigma b(\bw, c_2, \mu_2), \label{eq:cont-err-eq-mom} \\
                	(q, \divG \obu) & = 0, \label{eq:cont-err-eq-div} \\
                	\langle \pa_t \oc, v \rangle + M a(\omu, v) & = b(\bu_2, c_2, v) - b(\bu_1, c_1, v), \label{eq:cont-err-eq-phase} \\
                	(\omu, g) & = \varepsilon^{-1} (\ophi, g) + \varepsilon a(\oc, g). \label{eq:cont-err-eq-chem}
                \end{align}
            \end{subequations}
            We test the error equations \eqref{eq:cont-err-eq} with $(\bw, q, v) = (\obu, \op, \sigma \omu)$ to deduce:
            \begin{equation} \label{eq:cont-err-eq-tested-strong}
                \begin{split}
                	2 \eta \|D_\G \obu\|^2 & = \sigma b(\obu, c_1, \omu) + \sigma b(\obu, \oc, \mu_2), \\
                	\sigma \langle \pa_t \oc, \omu \rangle + \sigma M \|\nablaG \omu\|^2 & = - \sigma b(\obu, c_1, \omu) - \sigma b(\bu_2, \oc, \omu).
                \end{split}
            \end{equation}
            Next, we note that since $\mu_i, f_0'(c_i), \DeltaG c_i \in L^2(I; H^1(\G))$, the density of $H^1(\G)$ in $H^1(\G)'$ implies that the chemical potential error equation is valid in the following form for a.e. $t \in I$,
            \begin{equation} \label{eq:chem-pot-strong}
                \langle g, \omu \rangle = \varepsilon^{-1} \langle g, \ophi \rangle - \varepsilon \langle g, \DeltaG \oc \rangle, \qquad \forall g \in H^1(\G)'.
            \end{equation}
            {Testing~\eqref{eq:chem-pot-strong} with $g=\sigma\pa_t\oc$ and using the chain rule $\langle\pa_t\oc,-\DeltaG\oc\rangle=\tfrac12\ddt\|\nablaG\oc\|^2$, justified in the proof of Theorem~\ref{thm:well-posedness}, yields}
            \begin{equation} \label{eq:cont-chem-pot-eq-tested-strong}
                -\sigma \langle \pa_t \oc, \omu \rangle = - \sigma \varepsilon^{-1} \langle \pa_t \oc, \ophi \rangle - \frac{\sigma \varepsilon}{2} \ddt \|\nablaG \oc\|^2. 
            \end{equation}
            Adding the three equations in \eqref{eq:cont-err-eq-tested-strong} and \eqref{eq:cont-chem-pot-eq-tested-strong} noting the cancellation of $b(\obu, c_1, \omu)$ yields
            \begin{equation} \label{eq:err-eq-strong-prelim}
                \frac{\sigma \varepsilon}{2} \ddt \|\nablaG \oc\|^2 + 2 \eta \|D_\G \obu\|^2 + \sigma M \|\nablaG \omu\|^2 = \sigma b(\obu, \oc, \mu_2) - \sigma b(\bu_2, \oc, \omu) - \sigma \varepsilon^{-1} \langle \pa_t \oc, \ophi \rangle.
            \end{equation}
            {To handle the last term in \eqref{eq:err-eq-strong-prelim}, the \emph{key trick} is to test the phase field error equation \eqref{eq:cont-err-eq-phase} with $\ophi = f_0'(c_1) - f_0'(c_2)$, which is valid {by Lemma~\ref{lem:reg-weak}}}:
            \begin{equation} \label{eq:chem-pot-ophi}
                \langle \pa_t \oc, \ophi \rangle = - b(\obu, c_1, \ophi) - b(\bu_2, \oc, \ophi) - M a(\omu, \ophi).
            \end{equation}
            Combining \eqref{eq:err-eq-strong-prelim} with \eqref{eq:chem-pot-ophi} yields
            \begin{equation} \label{eq:err-eq-strong}
                \begin{split}
                    \frac{\sigma \varepsilon}{2} \ddt \|\nablaG \oc\|^2 & + 2 \eta \|D_\G \obu\|^2 + \sigma M \|\nablaG \omu\|^2 \\
                    & = \sigma b(\obu, \oc, \mu_2) - \sigma b(\bu_2, \oc, \omu) + \sigma \varepsilon^{-1} [b(\obu, c_1, \ophi) + b(\bu_2, \oc, \ophi) + M a(\omu, \ophi)] \\
                    & =: \rJ_1 + \rJ_2 + \rJ_3 + \rJ_4 + \rJ_5.
                \end{split}
            \end{equation}
            Integrating \eqref{eq:err-eq-strong} from $0$ to $t$ yields
            \begin{equation} \label{eq:err-eq-strong-integrated}
                \begin{split}
                    \frac{\sigma \varepsilon}{2} \|\nablaG \oc(t)\|^2 & + 2 \eta \|D_\G \obu\|_{L^2([0,t]; \dL^2(\G))}^2 + \sigma M \|\nablaG \omu\|_{L^2([0,t]; \bL^2(\G))}^2 \\
                    & = \frac{\sigma \varepsilon}{2} \|\nablaG \oc_0\|^2 + \int_0^t \rJ_1 + \rJ_2 + \rJ_3 + \rJ_4 + \rJ_5 \ds.
                \end{split}
            \end{equation}
            Recall that $\wh{c}_0 := |\G|^{-1} (c_{0,1} - c_{0,2},1)$ and note that $|\G|^{-1} (\oc, 1) = |\G|^{-1} (\oc_0, 1) = \wh{c}_0$ whence $|\mean{\oc}| \leq |\G|^{-1/2} \|\oc_0\|$. Similarly, from the chemical potential equation~\eqref{eq:weak-chem}, we infer $|\mean{\omu}| = \varepsilon^{-1} |\mean{\ophi}| \leq \varepsilon^{-1} |\G|^{-1/2} \|\ophi\| \leq C \|\oc\|_{H^1(\G)}$ from Lemma \ref{lem:control-of-nonlinearity} (control of the nonlinearity). Poincar\'e inequality \eqref{eq:poincare} together with $\nablaG v = \nablaG (v - \mean{v})$ for $v \in H^1(\G)$ implies that
            \begin{equation} \label{eq:equiv-norm-H1G}
                \|v\|_{H^1(\G)}^2 = \|v - \mean{v}\|^2 + \|\mean{v}\|^2 + \|\nablaG v\|^2 \leq |\G| \mean{v}^2 + (1 + C_P^2) \|\nablaG v\|^2.
            \end{equation}
            Therefore, the seminorms on $\oc, \omu$ can be upgraded to the full $H^1(\G)$-norms provided we balance it on the right-hand side:
            \begin{equation} \label{eq:err-eq-strong-h1-norms}
                \begin{split}
                    \frac{\sigma \varepsilon}{2} \|\oc(t)\|_{H^1(\G)}^2 & + 2 \eta \|D_\G \obu\|_{L^2([0,t]; \dL^2(\G))}^2 + \sigma M \|\omu\|_{L^2([0,t]; H^1(\G))}^2 \\
                    & \leq C_1 \|\oc_0\|_{H^1(\G)}^2 + C_2 \int_0^t \|\oc\|_{H^1(\G)}^2 \ds + \int_0^t |\rJ_1 + \rJ_2 + \rJ_3 + \rJ_4 + \rJ_5| \ds,
                \end{split}
            \end{equation}
            where $C_1 > 0$ depends on $\G$ and material parameters, while $C_2 > 0$ depends on $\G$, material parameters and $c_1, c_2$. We next recall that the trilinear form $b$ is continuous and estimate each of $\rJ_i$:
            \[
            \begin{split}
                |\rJ_1| & \leq C \|\obu\|_{\bH^1(\G)} \|\oc\|_{H^1(\G)} \|\mu_2\|_{H^1(\G)} \leq \delta_1 \|\obu\|_{\bH^1(\G)}^2 + C_{\delta,1} \|\oc\|_{H^1(\G)}^2 \|\mu_2\|_{H^1(\G)}^2, \\
                |\rJ_2| & \leq C \|\bu_2\|_{\bH^1(\G)} \|\oc\|_{H^1(\G)} \|\omu\|_{H^1(\G)} \leq \delta_2 \|\omu\|_{H^1(\G)}^2 + C_{\delta,2} \|\bu_2\|_{\bH^1(\G)}^2 \|\oc\|_{H^1(\G)}^2, \\
                |\rJ_3| & \leq C \|\obu\|_{\bH^1(\G)} \|c_1\|_{H^1(\G)} \|\ophi\|_{H^1(\G)} \leq \delta_3 \|\obu\|_{\bH^1(\G)}^2 + C_{\delta, 3} \|c_1\|_{L^\infty(I; H^1(\G))}^2 \|\ophi\|_{H^1(\G)}^2, \\
                |\rJ_4| & \leq C \|\bu_2\|_{\bH^1(\G)} \|\oc\|_{H^1(\G)} \|\ophi\|_{H^1(\G)} \leq C \|\bu_2\|_{\bH^1(\G)}^2 \|\oc\|_{H^1(\G)}^2 + \|\ophi\|_{H^1(\G)}^2, \\
                |\rJ_5| & \leq \delta_5 \|\omu\|_{H^1(\G)}^2 + C_{\delta, 5} \|\ophi\|_{H^1(\G)}^2. 
            \end{split}
            \]
            Using Korn's inequality \eqref{eq:korn} in \eqref{eq:err-eq-strong-h1-norms}, and choosing $\delta_i > 0$ sufficiently small yields
            \begin{equation} \label{eq:err-bnd-strong}
                \begin{split}
                    \|\oc(t)\|_{H^1(\G)}^2 & + \|\obu\|_{L^2([0,t]; \bH^1(\G))}^2 + \|\omu\|_{L^2([0,t]; H^1(\G))}^2 \\
                    & \leq C_3 \|\oc(0)\|_{H^1(\G)}^2 + C_4 \int_0^t\{(\|\mu_2\|_{H^1(\G)}^2 + \|\bu_2\|_{\bH^1(\G)}^2) \|\oc\|_{H^1(\G)}^2 + \|\ophi\|_{H^1(\G)}^2\}  \ds,
                \end{split}
            \end{equation}
            where $C_3 > 0$ depends on $\G$ and material parameters, and $C_4 > 0$ depends on $\G, c_1, \bu_2$ and material parameters. Using Lemma \ref{lem:control-of-nonlinearity} (control of nonlinearity) yields
            \[
                \begin{split}
                    \|\oc(t)\|_{H^1(\G)}^2 & + \|\obu\|_{L^2([0,t]; \bH^1(\G))}^2 + \|\omu\|_{L^2([0,t]; H^1(\G))}^2 \leq C \|\oc(0)\|_{H^1(\G)}^2 \\
                    & \quad + C \int_0^t (C_{\phi,1}(s) + \|\mu_2\|_{H^1(\G)}^2 + \|\bu_2\|_{\bH^1(\G)}^2) \|\oc\|_{H^1(\G)}^2 \ds,
                \end{split}
            \]
            where $C_{\phi,1}(t) \in L^1(I)$ is given in \eqref{eq:Cphi1} with $\xi_i = c_i$. Since $C_{\phi,1}(s) + \|\mu_2\|_{H^1(\G)}^2 + \|\bu_2\|_{\bH^1(\G)}^2 \in L^1(I)$ due to a priori bounds on the two solutions, Gr\"onwall's inequality readily implies that
            \[
                \|\oc\|_{L^\infty(I; H^1(\G))}^2 + \|\obu\|_{L^2(I; \bH^1(\G))}^2 + \|\omu\|_{L^2(I; H^1(\G))}^2 \leq C_T \|\oc(0)\|_{H^1(\G)}^2,
            \]
            where $C_T$ depends on $\G, f_0$, the regularity of the two solutions (controlled in turn by $\cE[c_{0,i}]$ and means $\mean{c_{0,i}}$), material parameters, and grows exponentially with $T$. The bound on $\op$ can be recovered from the surjectivity of $\divG : \bH^1_{t, \perp}(\G) \to L^2_\#(\G)$. The bound on $\pa_t \oc$ can be recovered from the error equation for $\oc$ in \eqref{eq:cont-err-eq-phase}. We omit further details. This finishes the proof.
        \end{proof}

        \begin{theorem}[well-posedness of the weak formulation] \label{thm:well-posedness}
            Let $\G$ be of class $C^2$ and $d \in \{2,3\}$. Then for each $c_0 \in H^1(\G)$, there exists a unique weak solution $(\bu, p, c, \mu)$ of \eqref{eq:weak}. Moreover, this weak solution satisfies the energy \emph{equality} \eqref{eq:en-law} and mass conservation property $(c(t) - c_0, 1) = 0$ for all $t \in [0,T]$.
        \end{theorem}
        \begin{proof}
            Uniqueness of weak solutions of \eqref{eq:weak} follows immediately from Theorems \ref{thm:existence} and \ref{thm:cont-dep} upon choosing $c_{0,1} = c_{0,2} = c_0 \in H^1(\G)$ in Theorem \ref{thm:cont-dep}. {The} energy \emph{inequality} version of \eqref{eq:en-law} follows from Theorem \ref{thm:existence}. We now upgrade it to an \emph{equality}.

            From Lemma \ref{lem:reg-weak}  we know that $c \in L^\infty(I; H^1(\G)) \cap L^2(I; W^{2,6}(\G))$ and $f_0'(c), \DeltaG c \in L^2(I; H^1(\G))$. Note that the pure Cahn--Hilliard equation is the gradient flow of the energy with respect to the inner product $\langle \ell_1, (-\DeltaG)^{-1} \ell_2 \rangle$ on $H^{-1}_\#(\G) = \{\ell \in H^1(\G)' : \langle \ell, 1 \rangle = 0\}$. On $H^{-1}_\#(\G)$ the Dirichlet energy $\varphi(v) := \frac12 \|\nablaG v\|^2$ for $v \in H^1_\#(\G)$, extended by $+\infty$, is convex, proper, and lower semicontinuous, and $\DeltaG^2 v \in \pa \varphi(v)$ whenever $\DeltaG v \in H^1(\G)$. Since $\langle \pa_t c, 1 \rangle = 0$ and $\DeltaG c \in L^2(I; H^1(\G))$, the chain rule \cite[Lemme~3.3, p.~73]{Brezis1973} applied to $\zm{c} \in H^1(I; H^{-1}_\#(\G))$ shows that $t \mapsto \|\nablaG c(t)\|^2$ is absolutely continuous and $\frac12 \ddt \|\nablaG c\|^2 = \langle \pa_t c, -\DeltaG c \rangle$ for a.e. $t \in I$. 
            
            Next note that the chain rule $\ddt (f_0(c), 1) = \langle \pa_t c, f_0'(c) \rangle$ is also valid; this can be verified by mollifying $c$ in time and using regularity provided by Lemma \ref{lem:reg-weak} to pass to the limit. Then $\sigma \langle \pa_t c, \mu \rangle = \sigma \varepsilon^{-1} \langle \pa_t c, f_0'(c) \rangle - \sigma \varepsilon \langle \pa_t c, \DeltaG c \rangle = \ddt \cE[c]$ for every weak solution. Testing \eqref{eq:weak-mom} with $\bw = \bu$ and \eqref{eq:weak-phase} with $v = \sigma \mu$ then shows that every weak solution satisfies \eqref{eq:en-law} exactly. This completes the proof.
        \end{proof}
        
    \subsection{Higher regularity} \label{subsec:regularity}
        We now turn our attention to the higher regularity statement for $d=2$. For the proof of the main result, we shall need the following maximal $L^2$-regularity property for $\DeltaG^2$ whose proof is given in the Appendix.
        \begin{proposition}[maximal parabolic $L^2$-regularity] \label{prop:maxreg-L2}
            Let $\G$ be of class $C^{m+4}$, $m \geq 0$, and $\varrho > 0$. Then the map $L^2(I; H^{m+4}(\G)) \cap H^1(I; H^m(\G)) \ni u \mapsto (\pa_t u + \varrho \DeltaG^2 u, u(0)) \in L^2(I; H^m(\G)) \times H^{m+2}(\G)$ is a linear homeomorphism with the stability constants depending only on $\G, \varrho, m$ and $T$.
        \end{proposition}
         {By Lemma~\ref{lem:reg-weak}, Proposition~\ref{prop:cap-force} extends to weak solutions by density.} We sketch the proof of the higher regularity result.
        \begin{theorem}[higher regularity] \label{thm:regularity}
            Let $\G$ be of class $C^{m+4}$, $m \geq 0$, $d = 2$, and $c_0 \in H^{m+2}(\G)$. Suppose additionally that $f_0 \in C^{m+2,1}_{loc}(\bbR)$. Then the weak solution $(\bu, p, c, \mu)$ of \eqref{eq:weak} satisfies
            \begin{equation} \label{eq:spat-reg}
                \begin{split}
                    c \in C([0,T]; H^{m+2}(\G)) \cap H^1(I; H^m(\G)) \cap L^2(I; H^{m+4}(\G)), \qquad \mu \in L^2(I; H^{m+2}(\G)), \\
                    \bu \in L^2(I; \bH^{m+3}_{t,\perp}(\G)), \qquad p \in L^2(I; H^{m+2}_\#(\G)). \qquad \qquad \qquad \qquad
                \end{split}
            \end{equation}
          For $m=0$, the norms of $c, \mu, \bu, p$ in these spaces are bounded by a constant depending only on $\G$, $T$, material parameters, $\|c_0\|_{H^{m+2}(\G)}$, $f_0(0), f_0'(0), f''_0(0)$, and $L$ from \eqref{eq:double-well}. For $m \geq 1$, the constant also depends on $\|f_0\|_{C^{m+2,1}([-R, R])}$ where $R \geq \|c\|_{L^\infty(\G \times I)}$; $R$ is controlled by the estimate for $m=0$ since $c \in C^0([0,T]; H^2(\G)) \subseteq L^\infty([0,T] \times \G)$. 
        \end{theorem}
        \begin{proof}
            \emph{Step 1: Stokes regularity for $(\bu, p)$}. Recalling that $\mu \nablaG c \perp E$, the test space for the momentum equation can be extended to $\bH^1_t(\G)$. Using $c \in L^\infty(I; H^1(\G))$ in \eqref{eq:weak-mom} we find that $\zm{\mu} \nablaG c \in L^2(I; \bL^{4/3}(\G))$. Then by higher Stokes regularity \cite[Theorem 6.1]{BenavidesNochettoShakipov2025-vector} valid for $\G$ of class $C^3$, we infer that $(\bu, p) \in L^2(I; \bW^{2, \frac43}(\G)) \times L^2(I; W^{1, \frac43}(\G))$. The important  observation here is that $\bu \in L^2(I; \bC^0(\G))$ by the Sobolev embedding $\bW^{2, \frac43}(\G) \hookrightarrow \bC^0(\G)$ valid for $d=2$ (but not for $d=3$).  This implies $\bu\cdot\nablaG c\in L^2(I;L^2(\G))$ because $\nablaG c\in L^\infty(I;\bL^2(\G))$.

            \emph{Step 2: Elliptic regularity for $c$}. By Lemma \ref{lem:reg-weak} (regularity of a weak solution), we know that {$f_0'(c),\mu\in L^2(I;H^1(\G))$}. Therefore, elliptic regularity for $\G$ of class $C^3$ applied to \eqref{eq:weak-chem} implies $c \in L^2(I; H^3(\G))$.
            
            \emph{Step 3: Maximal parabolic regularity for $c$}. With this, we combine \eqref{eq:weak-phase} and \eqref{eq:weak-chem} to find that $c \in L^2(I; H^3(\G)) \cap H^1(I; H^1(\G)')$ satisfies $\pa_t c + M \varepsilon \DeltaG^2 c = F$ in $L^2(I; H^1(\G)')$ where $F = M \varepsilon^{-1} \DeltaG f_0'(c) - \bu \cdot \nablaG c$. From Step 1, we already know that $\bu \cdot \nablaG c \in L^2(I; L^2(\G))$.
            
            For the other term, we compute $\DeltaG f_0'(c) = \divG (f_0''(c) \nablaG c) = f_0'''(c) |\nablaG c|^2 + f_0''(c) \DeltaG c$. Term $f_0''(c) \DeltaG c$ can be estimated using \eqref{eq:double-well} as follows: $\|f_0''(c) \DeltaG c\| \leq \|f_0''(c)\|_{L^4(\G)} \|\DeltaG c\|_{L^4(\G)} \leq C (1 + \|c\|_{L^8(\G)}^2) \|c\|_{H^3(\G)} \leq C (1 + \|c\|_{L^\infty(I; H^1(\G))}^2) \|c\|_{H^3(\G)}$, which is in $L^2(I)$. For $f_0'''(c) |\nablaG c|^2$, we first note that since $c \in L^\infty(I; H^1(\G)) \cap L^2(I; H^3(\G))$, elliptic regularity together with
            \[
                \begin{split}
                    \int_I \|\DeltaG c\|^4 \dt & \leq \int_I \|\DeltaG c\|_{H^1(\G)'}^2 \|\DeltaG c\|_{H^1(\G)}^2 \dt \leq \int_I \|c\|_{H^1(\G)}^2 \|c\|_{H^3(\G)}^2 \dt \\
                    & \leq \|c\|_{L^\infty(I; H^1(\G))}^2 \|c\|_{L^2(I; H^3(\G))}^2
                \end{split}
            \]
            readily yields that $c \in L^4(I; H^2(\G))$. Using \eqref{eq:double-well} together with Gagliardo--Nirenberg interpolation inequality, we infer
            \[
                \begin{split}
                    \|f_0'''(c) |\nablaG c|^2\| & \leq \|f_0'''(c)\|_{L^6(\G)} \|\nablaG c\|_{\bL^6(\G)}^2 \leq C (1 + \|c\|_{L^6(\G)}) \|c\|_{W^{1,6}(\G)}^2 \\
                    & \leq C (1 + \|c\|_{L^\infty(I; H^1(\G))}) \|c\|_{H^1(\G)}^{2/3} \|c\|_{H^2(\G)}^{4/3} \leq C(c) \|c\|_{H^2(\G)}^{4/3}.
                \end{split}
            \]
            Taking the $L^2(I)$-norm on both sides and using $c \in L^4(I; H^2(\G))$ yields $f_0'''(c) |\nablaG c|^2 \in L^2(I; L^2(\G))$, and so is $F = M \varepsilon^{-1} \DeltaG f_0'(c) - \bu \cdot \nablaG c$.

            By using Proposition \ref{prop:maxreg-L2} with $m=0$ we infer that $c \in L^2(I; H^4(\G)) \cap H^1(I; L^2(\G))$. Then, it is easy to see that $\mu \in L^2(I; H^2(\G))$. 
            
            \emph{Step 4: Stokes regularity for $(\bu, p)$ again}. Finally, going back to the Stokes block, we find that $\mu \nablaG c \in L^2(I; \bH^1_t(\G))$ so that $(\bu, p) \in L^2(I; \bH^3(\G)) \times L^2(I; H^2(\G))$. 
            
            Higher spatial regularity for $m \geq 1$ can be deduced by bootstrapping using Proposition \ref{prop:maxreg-L2}, provided that $f_0$ is sufficiently regular.  We omit further details. 
        \end{proof}

        \begin{remark}[higher regularity in time]
        With sufficiently smooth data, higher regularity in time can also be obtained by bootstrapping. The resulting statements become cumbersome because $(c,\mu)$ and $(\bu,p)$ gain temporal regularity at different rates, and we do not pursue them here. Since the surface is closed and stationary, there are no boundary compatibility conditions, and only $c$ requires independent initial data.
        \end{remark}


\section{Discretization} \label{sec:method}
    In this section, we propose a structure-preserving trace finite element method  (TraceFEM) for the surface Stokes--Cahn--Hilliard system \eqref{eq:sch}. TraceFEM is a  finite element method for solving PDEs defined on embedded surfaces by utilizing tessellations of the ambient space~\cite{OlshanskiiReuskenGrande2009, OlshanskiiReusken2017, BurmanHansboLarsonMassing2018}.

    \subsection{TraceFEM} \label{subsec:tracefem}
        Let $\Omega$ be a polyhedral domain containing $\G$ given by the zero of a sufficiently smooth level-set function $\phi : \Omega \to \bbR$. Let $\{\wt \cT_h\}_{h>0}$ be a family of quasi-uniform  triangulations of $\Omega$ with characteristic mesh size $h$. For practical purposes we assume 
        a piecewise polynomial of degree $k_\phi \geq 1$ continuous approximation of $\phi$ on {$\wt\cT_h$}, denoted by $\phi_h$, such that
        \begin{equation} \label{eq:phi}
        \|\phi - \phi_h\|_{L^\infty(\Omega)} + h \|\nabla(\phi - \phi_h)\|_{L^\infty(\Omega)} \leq C h^{k_\phi + 1}.
        \end{equation}
For example, $\phi_h$ {can be} a Lagrange interpolant of $\phi$ of order $k_\phi \geq 1$ or can be recovered as a numerical solution of an equation that defines $\Gamma$ implicitly. 
        
    We define the discrete surface as $\G_h := \{\bx \in \Omega : \phi_h(\bx) = 0\}$. 
 If $k_\phi=1$, then integration over level sets of $\phi_h$ is straightforward. Otherwise one may resort  to the isoparametric TraceFEM \cite{Lehrenfeld2016, LehrenfeldReusken2018}. 
  Thus, $\G_h$ is a $C^0$ piecewise smooth surface not fitted to the background mesh $\wt \cT_h$. We define the active mesh and the computational domain as
        \begin{equation} \label{eq:cTh}
            \cT_h := \{T \in \wt \cT_h : T \cap \G_h \neq \emptyset\}, \quad \Omega_h := \cup_{T \in \cT_h} T.
        \end{equation}
        We shall use $(\cdot, \cdot)_{\Xi}$ to denote the $L^2(\Xi)$-inner product for scalar-, vector-, and tensor-valued functions, and $\|\cdot\|_{\Xi} := \|\cdot\|_{L^2(\Xi)}$ for $\Xi \in \{\Omega_h, \G_h\}$.
    
        Let $ \bV_h \subseteq \bH^1(\Omega_h), Q_h \subseteq H^1(\Omega_h)$, and $V_h \subseteq H^1(\Omega_h)$ be families of Lagrange elements over $\cT_h$ of degrees $k_p + 1$, $k_p$, and $k_c$, respectively. We consider $k_p, k_c \geq 1$ and let 
        $$Q_{h, \#} := \{q_h \in Q_h: (q_h, 1)_{\G_h} = 0\}, \quad V_{h, \#} := \{v_h \in V_h: (v_h, 1)_{\G_h} = 0\}.$$ 
        
        We define $\bn_h := \frac{\nabla \phi_h}{|\nabla \phi_h|}$, a piecewise smooth unit normal to $\G_h$, $\bP_h := \bI - \bn_h \otimes \bn_h$, $\bB_h := \bP_h \operatorname{symm}(\nabla \bn_h) \bP_h$ elementwise. Furthermore, we let $\wt \bn_h$ be an approximation of $\bn$ one order higher than $\bn_h$; this is needed for optimal accuracy of the Stokes block \cite{HansboLarsonLarsson2020, JankuhnOlshanskiiReuskenZhiliakov2021}. 
        
        Lastly, we let $\{\be_{h,i}\}_{i=1}^{\dim E} \subseteq \bP_h \bV_h$ be the $\bL^2(\G_h)$-orthonormal tangential basis for the discrete Killing space $E_h$; its computation will be discussed {in Section~\ref{subsec:Killing-identif}}. We then define $\cP_h$ to be the $\bL^2(\G_h)$-orthogonal  projection  onto $E_h$.  
        
    \subsection{Stabilization and bilinear forms} \label{subsec:stabilization}
        For $l=0,1$, we define the following stabilization and penalty terms: for all $u_h, v_h \in V_h$ or $Q_h$, and all $\bu_h, \bv_h \in \bV_h$,
        \begin{equation} \label{eq:stab}
            \begin{split}
                & s_l^{\mathrm{ext}}(u_h, v_h) := h^{1-2l} (\bn_h \cdot \nabla u_h, \bn_h \cdot \nabla v_h)_{\Omega_h}, \\
                & \bs_l^{\mathrm{ext}}(\bu_h, \bv_h) := h^{1-2l} ((\nabla \bu_h) \bn_h, (\nabla \bv_h) \bn_h)_{\Omega_h}, \\
                & \bs^{\mathrm{tan}}(\bu_h, \bv_h) := h^{-2} (\wt \bn_h \cdot \bu_h, \wt \bn_h \cdot \bv_h)_{\G_h}.
            \end{split}
        \end{equation}
        To control the Killing modes of the discrete velocity, we {include} the additional Killing penalty term:
        \begin{equation}
            \bs^{\mathrm{Kil}}(\bu_h, \bv_h) := h^{-2} (\cP_h  \bu_h, \cP_h  \bv_h)_{\G_h}.
        \end{equation}
        The extension terms stabilize the $H^l$ inner products. We next define the following bilinear forms: for all $u_h, v_h \in V_h$ or $Q_h$, and all $\bu_h, \bv_h \in \bV_h$,
        \begin{equation} \label{eq:bilinear-forms}
            \begin{split}
                (u_h, v_h)_h & := (u_h, v_h)_{\G_h} + s_0^{\mathrm{ext}}(u_h, v_h), \\
                a_h(u_h, v_h) & := (\nabla_{\G_h} u_h, \nabla_{\G_h} v_h)_{\G_h} + s_1^{\mathrm{ext}}(u_h, v_h), \\
                \ba_{\G_h}(\bu_h, \bv_h) & := (D_{\G_h} \bu_h - (\bu_h \cdot \bn_h) \bB_h, D_{\G_h} \bv_h - (\bv_h \cdot \bn_h) \bB_h)_{\G_h}, \\ 
                \ba_h(\bu_h, \bv_h) & := \ba_{\G_h}(\bu_h, \bv_h) + \bs^{\mathrm{tan}}(\bu_h, \bv_h) + \bs_1^{\mathrm{ext}} (\bu_h, \bv_h) + \bs^{\mathrm{Kil}}(\bu_h, \bv_h),
            \end{split}
        \end{equation}
        Here, $\nabla_{\G_h} = \bP_h \nabla$ and $D_{\G_h} = \bP_h \operatorname{symm} (\nabla) \bP_h$. The expression $D_{\G_h} \bv_h - (\bv_h \cdot \bn_h) \bB_h$ is a consistent replacement for $D_{\G_h} (\bP_h \bv_h)$ (since $D_\G \bv = D_\G (\bP \bv) + (\bv \cdot \bn) \bB$), which avoids differentiating the discontinuous projection; cf. \cite{HansboLarsonLarsson2020}.

    \subsection{The semidiscrete method} \label{subsec:semidiscrete}
        TraceFEM for the surface Stokes--Cahn--Hilliard problem reads: given $c_0$ with $\cE[c_0] < \infty$, find $(\bu_h, p_h, c_h, \mu_h) \in L^2(I; \bV_h) \times L^2(I; Q_{h, \#}) \times H^1(I; V_h) \times L^2(I; V_h)$ such that for a.e. $t \in (0,T]$ and all $(\bw_h, q_h, v_h, g_h) \in \bV_h \times Q_{h, \#} \times V_h \times V_h$,
        \begin{equation} \label{eq:semidiscrete}
            \begin{split}
                2 \eta \, \ba_h(\bu_h, \bw_h) + (\nabla_{\G_h} p_h, \bw_h)_{\G_h} & = \sigma b_h(\cP_h^\perp   \bw_h, c_h, \mu_h), \\
                (\nabla_{\G_h} q_h, \bu_h)_{\G_h} - s_0^{\mathrm{ext}}(p_h, q_h) & = 0, \\
                (\pa_t c_h, v_h)_h + b_h(\cP_h^\perp   \bu_h, c_h, v_h) + M a_h(\mu_h, v_h) & = 0, \\
                (\mu_h, g_h)_h & = \varepsilon^{-1} (f_0'(c_h), g_h)_{\G_h} + \varepsilon a_h(c_h, g_h),
            \end{split}
        \end{equation}
        and $c_h(0) = I_h c_0^e$, where $I_h : H^1(\Omega_h) \to V_h$ {is a suitable  projection or interpolation procedure.} Here, $b_h(\bw, v, g) := (\bw \cdot \nabla_{\G_h} v, g - \langle g \rangle_{\G_h})_{\G_h}$, the discrete counterpart  of $b$.

    \subsection{Structure preservation} \label{subsec:structure-preservation}
        Upon testing \eqref{eq:semidiscrete} with $\bw_h = \bu_h, q_h = p_h, v_h = \sigma \mu_h$ and $g_h = \sigma \pa_t c_h$, we find that the solution $(\bu_h, p_h, c_h, \mu_h)$ satisfies the following energy balance for $t \in (0, T]$:
        \begin{equation} \label{eq:discr-en-law}
            \ddt \cE_h[c_h(t)] + 2 \eta \, \ba_h(\bu_h, \bu_h) + s_0^{\mathrm{ext}}(p_h, p_h) + \sigma M a_h(\mu_h, \mu_h) = 0,
        \end{equation}
        where
        \begin{equation}
            \cE_h[v_h] := \sigma \left( \frac{\varepsilon}{2} a_h(v_h, v_h) + \varepsilon^{-1} (f_0(v_h), 1)_{\G_h} \right).
        \end{equation}
        Furthermore, upon testing the third equation in \eqref{eq:semidiscrete} with $v_h = 1 \in V_h$, we obtain the conservation of the total phase
        \begin{equation} \label{eq:discr-mass-conserv}
            \ddt (c_h, 1)_{\G_h} = -b_h(\cP_h^\perp   \bu_h, c_h, 1) = 0,
        \end{equation}
        where we crucially used the fact that $b_h$ has the centered form. 

        \begin{remark}[recommended choice of polynomial degrees] \label{rmk:degrees}
            To ensure equilibration of geometric and discretization errors, we recommend choosing $k_\phi = k_c = k_p+1$ in practice.
        \end{remark}

\begin{remark}[on handling Killing vector fields]
        The numerical treatment of Killing vector fields is a known issue in surface flow problems. A positive tangential friction term $\zeta\bu$, with $\zeta$ independent of the mesh size, removes the kernel of the surface Stokes operator. Retaining velocity inertia instead leads to an evolution problem which conserves the Killing field from the initial state. If neither modification is appropriate for the model, an alternative is the regularization and filtering approach of~\cite{BonitoDemlowLicht2020}. That paper introduced an artificial friction term, leading to a perturbed Stokes problem of the form
\[
    -2\eta\Delta_S\bu+\delta\bu+\nablaG p=\mathbf{f},
    \qquad \delta=h^\alpha,\quad 1\leq\alpha\leq2,
\]
and {identified} approximate Killing fields through a discrete Stokes eigenvalue problem. For $\alpha=2$, their lowest-order BDM method achieves optimal second-order $L^2$ accuracy after removal of the Killing component.

For the coupled surface Stokes--Cahn--Hilliard problem, artificial friction introduces additional dissipation and thus modifies the energy balance of the friction-free model. Furthermore, removing the Killing component only from the transport velocity generally leaves an extra term in the energy balance. Our method applies the same projection in the transport and capillary terms and penalizes the approximate Killing component of the velocity. This yields the cancellation between the coupling terms and {preserves the energy balance of the stabilized semidiscrete problem.}
\end{remark}

    \subsection{Fully discrete method} \label{subsec:fully-discrete}
        Let $0 = t^0 < t^1 < \dots < t^N = T$ be a partition of $[0,T]$ with steps $\tau_n := t^n - t^{n-1}$.  For a generic quantity $\zeta$ we write $\zeta^n \approx \zeta(t^n), \delta \zeta^n := \zeta^n - \zeta^{n-1}$, and $\dtau \zeta^n := \tau_n^{-1} \delta \zeta^n$. 
        
        Cahn--Hilliard dynamics typically exhibits a rapid stage of spinodal decomposition followed by slow coarsening, which motivates second-order time discretization with variable time steps. We use a midpoint scheme with a secant approximation of the potential derivative; see~\cite{DuNicolaides1991,GomezHughes2011} for related energy-stable discretizations of the Cahn--Hilliard equation. For the quartic potential $f_0(s)=(s^2-1)^2/4$, we set $c_h^{n+1/2}:=\tfrac12(c_h^n+c_h^{n+1})$ and define
\[
    \wh f(a,b):=\frac14(a+b)(a^2+b^2-2).
\]
This expression equals $(f_0(a)-f_0(b))/(a-b)$ for $a\neq b$ and satisfies $\wh f(b,b)=f_0'(b)$.

 The fully discrete method now reads: Given $c_h^n$, find
$(\bu_h^{n+1/2},p_h^{n+1/2},c_h^{n+1},\mu_h^{n+1/2})
\in \bV_h\times Q_{h,\#}\times V_h\times V_h$
such that
        \begin{equation} \label{eq:fd}
            \begin{split}
                2 \eta \, \ba_h(\bu_h^{n+1/2}, \bw_h) + (\nabla_{\G_h} p_h^{n+1/2}, \bw_h)_{\G_h} - \sigma b_h(\cP_h^\perp   \bw_h, c_h^{n+1/2}, \mu_h^{n+1/2}) & = 0, \\
                (\nabla_{\G_h} q_h, \bu_h^{n+1/2})_{\G_h} - s_0^{\mathrm{ext}}(p_h^{n+1/2}, q_h) & = 0, \\
                (\dtau c_h^{n+1}, v_h)_h + b_h(\cP_h^\perp   \bu_h^{n+1/2}, c_h^{n+1/2}, v_h) + M a_h(\mu_h^{n+1/2}, v_h) & = 0, \\
                (\mu_h^{n+1/2}, g_h)_h - \varepsilon^{-1} (\wh f(c_h^{n+1}, c_h^n), g_h)_{\G_h} - \varepsilon a_h(c_h^{n+1/2}, g_h) & = 0.
            \end{split}
        \end{equation}
 for all $(\bw_h,q_h,v_h,g_h)\in\bV_h\times Q_{h,\#}\times V_h\times V_h$. 
        Testing with $\bw_h = \bu_h^{n+1/2}$, $q_h = p_h^{n+1/2}$, $v_h = \sigma \mu_h^{n+1/2}$, $g_h = \sigma \dtau c_h^{n+1}$ and using
        \[
            \begin{split}
                (\wh f(c_h^{n+1}, c_h^n), \delta c_h^{n+1})_{\G_h} & = (f_0(c_h^{n+1}) - f_0(c_h^n), 1)_{\G_h}, \\
                a_h(c_h^{n+1/2}, \delta c_h^{n+1}) & = \tfrac12 \big[ a_h(c_h^{n+1}, c_h^{n+1}) - a_h(c_h^n, c_h^n) \big]
            \end{split}
        \]
        gives the \emph{exact} discrete energy equality,
        \begin{equation} \label{eq:fd2-energy}
            \dtau \cE_h[c_h^{n+1}] + 2 \eta \, \ba_h(\bu_h^{n+1/2}, \bu_h^{n+1/2}) + s_0^{\mathrm{ext}}(p_h^{n+1/2}, p_h^{n+1/2}) + \sigma M a_h(\mu_h^{n+1/2}, \mu_h^{n+1/2}) = 0.
        \end{equation}
 This identity holds for every $\tau_{n+1}>0$ for which a solution of~\eqref{eq:fd} exists. Thus, time discretization introduces no additional dissipation beyond that present in the semidiscrete energy law.
Furthermore, testing the phase equation with $v_h=1$ also gives total phase conservation:
\[
    (c_h^{n+1},1)_{\G_h}=(c_h^n,1)_{\G_h}
    =(c_h^0,1)_{\G_h}.
\]

\begin{remark}[Rannacher startup] \label{rem:rannacher-startup}
Midpoint time stepping provides weak damping of stiff dissipative modes, which can reduce accuracy for rough initial data. To damp these modes, one can use Rannacher startup~\cite{Rannacher1984}, replacing the first time interval by two backward Euler half-steps before switching to~\eqref{eq:fd}. Further initial intervals may be treated similarly. This procedure is known to preserve second-order convergence in classical diffusion settings. The startup steps conserve the total phase but do not satisfy the exact energy identity~\eqref{eq:fd2-energy}.
\end{remark}

        
\section{Numerical experiments} \label{sec:numerical-experiments}
    The numerical method for surfaces in $\bbR^3$ is implemented in \texttt{NGSolve/ngsxfem} \cite{Schoberl2009, Lehrenfeld2021}, which supports higher-order cut geometry. 

      \subsection{Identification of Killing vector fields}
    \label{subsec:Killing-identif}
    
    Explicit formulas for Killing fields are generally unavailable.
    We therefore supplement the method with a procedure for identifying
    approximate Killing fields. Throughout this subsection, we consider
    closed, connected surfaces $\G\subset\bbR^3$.
    
    We distinguish between the space $E$ of intrinsic Killing fields
    and the subspace induced by infinitesimal rigid motions of the
    ambient space. To define the latter, let
    \[
    \cR :=
    \{\br:\bbR^3\to\bbR^3:
    \br(\bx)=\ba+\bomega\times\bx,\quad
    \ba,\bomega\in\bbR^3\}
    \]
    and set
    \[
    E^{\mathrm{extr}} :=
    \{\br|_\G:\br\in\cR,\quad
    \br\cdot\bn=0\text{ on }\G\}.
    \]
    Since $\operatorname{symm}(\nabla\br)=\bzero$ for every
    $\br\in\cR$, a tangential restriction satisfies
    \[
    D_\G(\br|_\G)
    =\bP\operatorname{symm}(\nabla\br)\bP
    =\bzero.
    \]
    Hence $E^{\mathrm{extr}}\subseteq E$. {On a smooth closed surface with strictly positive Gaussian curvature, infinitesimal tangential rigidity implies that every tangential field with $D_\G\bv=0$ is the restriction of an ambient infinitesimal rigid motion; see~\cite[Vol.~V, Ch.~12, Thm.~10]{Spivak1999}. Hence $E^{\mathrm{extr}}=E$ for such surfaces.}

    This equality need not hold at lower regularity.
    Bartzos et al.~\cite[Section~7, Corollary~10]{BartzosEtAll2018}
    construct $C^1$ isometric embeddings of the round unit sphere
    into arbitrarily small balls in $\bbR^3$. Their images are
    compact, closed, connected surfaces and retain the
    three-dimensional intrinsic Killing algebra of the round
    metric. However they admit no ambient rotational symmetry and so their extrinsic Killing space is trivial.

   To the best of our knowledge,  whether $E^{\mathrm{extr}}=E$ holds for a general smooth  compact connected  surface without a boundary  embedded in $\mathbb{R}^3$ is {an} open question, and so one may choose to accept $E^{\mathrm{extr}}=E$ as an additional assumption.
    The first procedure below does not require that.
    Otherwise, if all Killing fields of the surface under consideration are
    known or assumed to be extrinsic, the second simpler procedure can
    replace the first.
    \smallskip
    
    \noindent
    \textbf{Elliptic eigenvalue problem (intrinsic $E$).}
    The kernel of the continuous bilinear form  $(D_\G\bu,D_\G\bv)_\G$ on tangential vector fields is precisely  $E$. This suggests identifying  discrete Killing fields   through an eigenvalue problem. Let
    $\{\bphi_i\}_{i=1}^{N_u}$ be the standard nodal basis of   $\bV_h$, and define
    \[
    \begin{aligned}
    	\bA_{ij}
    	&:=
    	\ba_{\G_h}(\bphi_i,\bphi_j)
    	+{\bs^{\mathrm{tan}}}(\bphi_i,\bphi_j)
    	+\bs_1^{\mathrm{ext}}(\bphi_i,\bphi_j),\\
    	\bM_{ij}
    	&:=
    	(\bphi_i,\bphi_j)_{\G_h}
    	+\bs_0^{\mathrm{ext}}(\bphi_i,\bphi_j).
    \end{aligned}
    \]
    Note that the Killing penalty is excluded from $\bA$.   The matrix $\bA$ is symmetric positive semidefinite, while     $\bM$ is positive definite. We solve
    \begin{equation}\label{eq:eigenpb}
    	\bA V=\lambda\bM V,
    	\qquad V\in\bbR^{N_u}\setminus\{0\}.
    \end{equation}
    No assumption that the discrete surface has a trivial Killing space is needed.
    
    Let $n_E=\dim E$. For a spectrally convergent approximation,  the first $n_E$ generalized eigenvalues converge to zero, whereas the next converges to the first positive eigenvalue
    of the continuous operator:
    \[
    \lambda_{h,j}\to0\quad (1\le j\le n_E),
    \qquad
    \lambda_{h,n_E+1}\to\lambda_{n_E+1}>0.
    \]
    This motivates identifying approximate Killing fields from    the cluster of eigenvalues tending to zero under mesh    refinement. Related spectral identification and filtering    for a surface Stokes discretization were analyzed    in~\cite[Section~6]{BonitoDemlowLicht2020}.
    Our eigenproblem does not impose a divergence constraint,    since the kernel of the continuous strain form already    consists of Killing fields, {and is therefore generally less expensive than a discrete Stokes eigenproblem.}  In Table~\ref{tab:eigs}, we examine the convergence of the eigenvalues of~\eqref{eq:eigenpb} under mesh refinement     for different orders of geometry approximation.

    Since $\dim E\le3$ for a connected two-dimensional surface,    we compute the four smallest generalized eigenvalues    and their eigenvectors. Lanczos iteration can be used when $\bA$ is symmetric positive semidefinite. 
    
    For each selected eigenvector $V^{(j)}$, let
    $
    \bz_{h,j}:=\sum_{i=1}^{N_u}V_i^{(j)}\bphi_i.
    $
    We define
    \[
    E_h:=\operatorname{span}
    \{\bP_h\bz_{h,j}:j\text{ as selected}\}
    \]
    and orthonormalize these fields in $\bL^2(\G_h)$.    This gives the tangential basis used to define $\cP_h$.

    If a sequence of refined grids is available, the separation of the small eigenvalues should 
    be checked against mesh refinement. Otherwise an ad hoc criterion, e.g., $\lambda_i \leq h^\ell$ for some $\ell>0$, can be used.  Since the surface is stationary, the selected space and  its projection are computed once before time stepping.
    \smallskip
    
    \noindent
    \textbf{Restriction of ambient rigid motions:
    	a $6$-by-$6$ eigenvalue problem (extrinsic).}
    If  $E^{\mathrm{extr}}=E$, it suffices to find the    infinitesimal rigid motions whose restrictions are tangential.    Let $\{\br_i\}_{i=1}^6$ be a basis of $\cR$ and define
    \[
    \wt\bA_{ij}
    :=(\br_i\cdot\wt\bn_h,\br_j\cdot\wt\bn_h)_{\G_h},
    \qquad
    \wt\bM_{ij}
    :=(\br_i,\br_j)_{\G_h}.
    \]
    The matrix $\wt\bA$ is symmetric positive semidefinite,    and $\wt\bM$ is positive definite.  
    
    We solve the generalized eigenvalue problem
    \begin{equation}\label{eq:eigenpb-small}
    	\wt\bA W=\lambda\wt\bM W,
    	\qquad W\in\bbR^6\setminus\{0\}.
    \end{equation}
    For $\br_W:=\sum_{i=1}^6W_i\br_i$, its Rayleigh quotient is
    \[
    \frac{W^T\wt\bA W}{W^T\wt\bM W}
    =
    \frac{\|\br_W\cdot\wt\bn_h\|_{\G_h}^2}
    {\|\br_W\|_{\G_h}^2}.
    \]
    Thus small eigenvalues identify ambient rigid motions    with small normal components. With exact geometry,    the zero eigenspace corresponds precisely to    $E^{\mathrm{extr}}$; with approximate geometry,    the associated eigenvalues tend to zero as the    geometry and normal approximation improve.
    
    The problem~\eqref{eq:eigenpb-small} can be solved directly.    For the selected eigenvectors, we tangentially project the corresponding fields $\br_W$ and orthonormalize their restrictions in $\bL^2(\G_h)$, as above.    When $E^{\mathrm{extr}}=E$, these fields provide an  alternative inexpensive construction of $E_h$. Otherwise, the procedure detects only the extrinsic subspace and    can be used to cross-check~\eqref{eq:eigenpb}.   A disagreement between the two procedures on a coarse mesh calls for refinement or independent information about the surface symmetries.

        \begin{table}{ \small 
            \centering
            \resizebox{\linewidth}{!}{
            \begin{tabular}{|c|c|c|cc|c|c|c|c|cc|}
            \hline
            \multirow{2}{*}{$\ell$} & \multirow{2}{*}{$h$} & \multirow{2}{*}{ndof} & \multicolumn{2}{c|}{$\lambda_1$} & \multirow{2}{*}{$\lambda_2$} & \multirow{2}{*}{$\lambda_3$} & \multirow{2}{*}{$\lambda_4$} & \multirow{2}{*}{$\lambda_2/\lambda_1$} & \multicolumn{2}{c|}{$\|\be_h - \be^{\rm true}\|_{L^2(\Gamma_h)}$} \\
            \cline{4-5}\cline{10-11}
                & & & value & rate & & & & & error & rate \\
            \hline
            \multicolumn{11}{|c|}{$k = 1$} \\
            \hline
            $\mathtt{0}$ & $\mathtt{4.65e\text{-}1}$ & $\mathtt{3.34e3}$ & $\mathtt{8.32e\text{-}3}$ & -- & $\mathtt{0.1578216}$ & $\mathtt{0.1589538}$ & $\mathtt{0.1610467}$ & $\mathtt{1.90e1}$ & $\mathtt{1.97e\text{-}2}$ & -- \\
            $\mathtt{1}$ & $\mathtt{2.33e\text{-}1}$ & $\mathtt{1.34e4}$ & $\mathtt{1.99e\text{-}3}$ & $\mathtt{2.06}$ & $\mathtt{0.1353458}$ & $\mathtt{0.1365392}$ & $\mathtt{0.1368615}$ & $\mathtt{6.82e1}$ & $\mathtt{4.81e\text{-}3}$ & $\mathtt{2.03}$ \\
            $\mathtt{2}$ & $\mathtt{1.17e\text{-}1}$ & $\mathtt{5.45e4}$ & $\mathtt{4.80e\text{-}4}$ & $\mathtt{2.02}$ & $\mathtt{0.1292626}$ & $\mathtt{0.1294376}$ & $\mathtt{0.1297776}$ & $\mathtt{2.69e2}$ & $\mathtt{1.17e\text{-}3}$ & $\mathtt{2.01}$ \\
            $\mathtt{3}$ & $\mathtt{5.82e\text{-}2}$ & $\mathtt{2.17e5}$ & $\mathtt{1.20e\text{-}4}$ & $\mathtt{2.01}$ & $\mathtt{0.1276289}$ & $\mathtt{0.1276670}$ & $\mathtt{0.1282280}$ & $\mathtt{1.07e3}$ & $\mathtt{3.01e\text{-}4}$ & $\mathtt{1.97}$ \\
            $\mathtt{4}$ & $\mathtt{2.91e\text{-}2}$ & $\mathtt{8.68e5}$ & $\mathtt{2.99e\text{-}5}$ & $\mathtt{2.00}$ & $\mathtt{0.1272158}$ & $\mathtt{0.1272259}$ & $\mathtt{0.1278617}$ & $\mathtt{4.26e3}$ & $\mathtt{7.62e\text{-}5}$ & $\mathtt{1.98}$ \\
            \hline
            \multicolumn{11}{|c|}{$k = 2$} \\
            \hline
            $\mathtt{1}$ & $\mathtt{2.33e\text{-}1}$ & $\mathtt{7.88e4}$ & $\mathtt{2.10e\text{-}5}$ & -- & $\mathtt{0.1271166}$ & $\mathtt{0.1271231}$ & $\mathtt{0.1277614}$ & $\mathtt{6.06e3}$ & $\mathtt{1.31e\text{-}4}$ & -- \\
            $\mathtt{2}$ & $\mathtt{1.17e\text{-}1}$ & $\mathtt{3.20e5}$ & $\mathtt{1.03e\text{-}6}$ & $\mathtt{4.30}$ & $\mathtt{0.1270807}$ & $\mathtt{0.1270811}$ & $\mathtt{0.1277427}$ & $\mathtt{1.24e5}$ & $\mathtt{1.16e\text{-}5}$ & $\mathtt{3.46}$ \\
            $\mathtt{3}$ & $\mathtt{5.82e\text{-}2}$ & $\mathtt{1.28e6}$ & $\mathtt{6.07e\text{-}8}$ & $\mathtt{4.09}$ & $\mathtt{0.1270786}$ & $\mathtt{0.1270787}$ & $\mathtt{0.1277413}$ & $\mathtt{2.09e6}$ & $\mathtt{1.26e\text{-}6}$ & $\mathtt{3.20}$ \\
            \hline
            \multicolumn{11}{|c|}{$k = 3$} \\
            \hline
            $\mathtt{1}$ & $\mathtt{2.33e\text{-}1}$ & $\mathtt{2.35e5}$ & $\mathtt{8.76e\text{-}7}$ & -- & $\mathtt{0.1270736}$ & $\mathtt{0.1270742}$ & $\mathtt{0.1277347}$ & $\mathtt{1.45e5}$ & $\mathtt{2.72e\text{-}5}$ & -- \\
            $\mathtt{2}$ & $\mathtt{1.17e\text{-}1}$ & $\mathtt{9.55e5}$ & $\mathtt{1.17e\text{-}8}$ & $\mathtt{6.15}$ & $\mathtt{0.1270783}$ & $\mathtt{0.1270783}$ & $\mathtt{0.1277410}$ & $\mathtt{1.08e7}$ & $\mathtt{1.63e\text{-}6}$ & $\mathtt{4.02}$ \\
            \hline
            \end{tabular}
            }
            \vskip1ex
            \caption{\small 
            	Intrinsic Killing fields identification using \eqref{eq:eigenpb} on the 2-1 torus. Here, $k$ is the order of the discrete Killing vector $\be_h$, level set $\phi_h$ and normal $\wt \bn_h$ appearing in ${\bs^{\mathrm{tan}}}$. 
            Since $\dim E = 1$ for the torus, we expect $\lambda_1$ to be $O(h^{2k})$ and $\lambda_2$ to be $O(1)$, which this experiment reproduces, allowing the Killing mode to be identified. The rates are computed with respect to ndof$^{1/2}$.}
            \label{tab:eigs}
        }
        \end{table}
        
We further illustrate the model and the method with two phase-separation experiments. The first considers a torus with constant material coefficients and examines the effects of mobility and viscosity on coarsening. The second is more practical, it uses a spherical liposome with experimentally informed parameters, phase-dependent viscosity, and degenerate mobility.
        
    \subsection{Phase separation on a torus} \label{subsec:torus}
\begin{figure}[ht]
	\centering
	\includegraphics[width=0.9\linewidth]{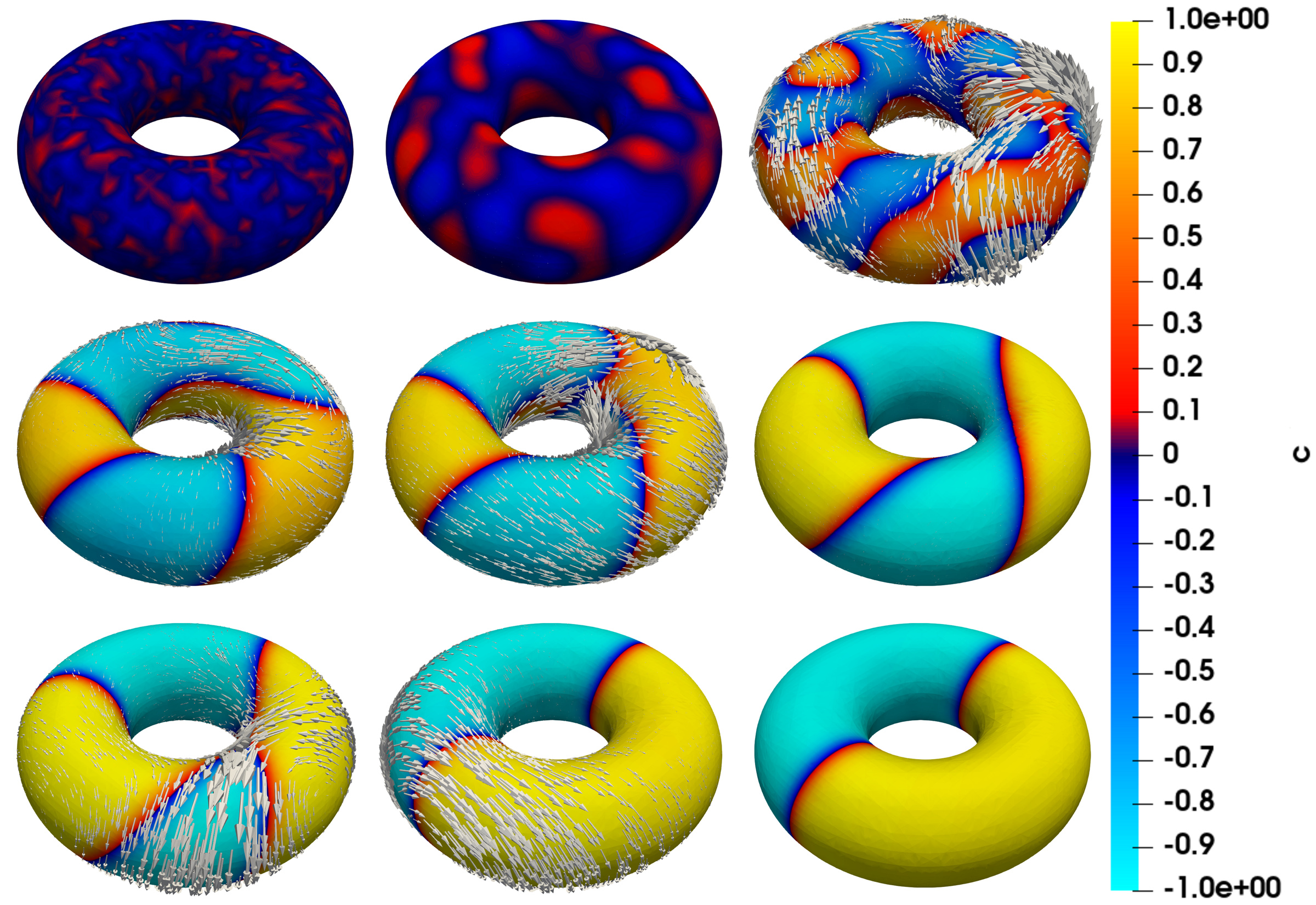}
	\caption{Phase separation and flow on the $2$-$1$ torus. Snapshots are ordered from left to right and top to bottom; colors show $c$ and arrows indicate the velocity.}
	\label{fig:torus}
\end{figure}

        We let $\G$ be the $2$-$1$ torus, $c_0$ be random noise {sampled from the uniform distribution on $[-0.1, 0.1]$} at the level of the mesh and set
        \[
            \eta=1, \quad \sigma=20, \quad M=0.02, \quad \varepsilon=0.15.
        \]
        This choice of parameters makes the interaction of the Stokes flow and the phase separation process more apparent. We use $(\bP_2,P_1,P_2,P_2)$ elements for $(\bu,p,c,\mu)$, with $h=0.2756$ and $138{,}000$ degrees of freedom, and the midpoint scheme with Rannacher startup and adaptive time stepping. Figure~\ref{fig:torus} shows the evolution up to $t=22$, with snapshots at $t=0$, $1.07$, $3.15$, $5.04$, $6.13$, $11.11$, $15.63$, $17.08$, and $22$. The initial small-scale fluctuations develop into larger domains. The snapshots around $t=6.13$ and $t=15.63$ show substantial flow near rearranging interfaces, followed by changes in domain connectivity. At later times, the evolution slows as the interface approaches a configuration that minimizes the interface perimeter under the total phase area conservation restriction.

        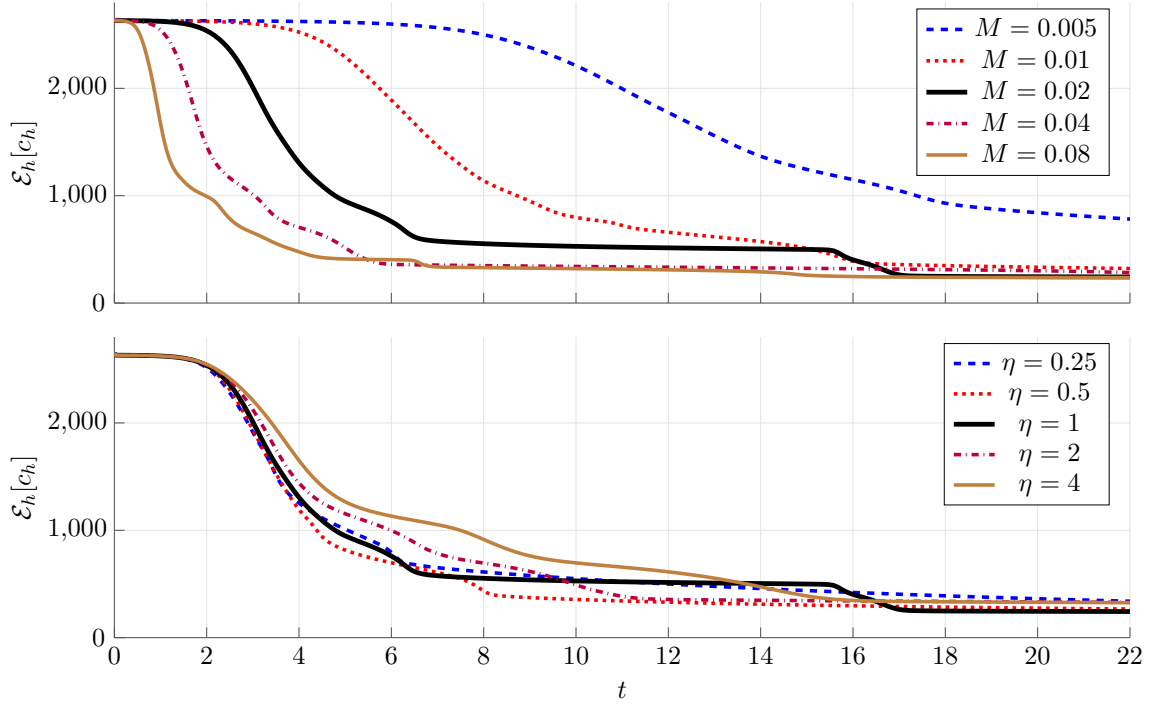
\begin{figure}
            \centering
            \scalebox{0.9}{%
            \begin{tikzpicture}
              \begin{groupplot}[
                  group style={
                    group size=1 by 2,
                    vertical sep=0.5cm,
                    xticklabels at=edge bottom,
                    xlabels at=edge bottom,
                  },
                  width=\linewidth,
                  xmin=0, xmax=22,
                  ymin=0, ymax=2800,
                  xlabel={$t$},
                  grid=major, major grid style={gridgray, line width=0.4pt},
                  axis lines*=left,
                  tick style={inkmuted}, axis line style={inkmuted},
                  yticklabel style={text width=2.2em, align=right},
                  ylabel style={at={(axis description cs:0,0.5)}, anchor=base, yshift=3.2em},
                  filter discard warning=false,
                ]
                
                \nextgroupplot[height=6cm, ylabel={$\cE_h[c_h]$}]
                \addrunM{1}{20}{0.005}{smlst}
                \addrunM{1}{20}{0.01}{smll}
                \addrunM{1}{20}{0.02}{midl}
                \addrunM{1}{20}{0.04}{larg}
                \addrunM{1}{20}{0.08}{huge}
             
                \nextgroupplot[height=6cm, ylabel={$\cE_h[c_h]$}]
                \addruneta{0.25}{20}{0.02}{smlst}
                \addruneta{0.5}{20}{0.02}{smll}
                \addruneta{1}{20}{0.02}{midl}
                \addruneta{2}{20}{0.02}{larg}
                \addruneta{4}{20}{0.02}{huge}

              \end{groupplot}
            \end{tikzpicture}%
            }

            \caption{Energy evolution on the torus for different mobilities $M$ (top) and viscosities $\eta$ (bottom).}
            \label{fig:torus-energy}
        \end{figure}
        
        Figure~\ref{fig:torus-energy} compares the energy histories when either $M$ or $\eta$ is varied with all other parameters fixed. All recorded energy histories are nonincreasing. Increasing mobility substantially accelerates the initial energy decay. Viscosity also affects the subsequent coarsening, but the curves do not remain ordered by viscosity, reflecting differences in the evolving domain configurations. Neither $M$ nor $\eta$ enters the energy functional; these parameters change the fluid and phases dynamics and the dissipation. The results show that the coarsening time and the route to lower-energy configurations depend on both phase diffusion and viscous transport.

    \subsection{A phase separated liposome} \label{subsec:liposome}
In this experiment, we consider a lipid composition studied in~\cite{WangPalzhanovQuainiOlshanskiiMajd2022lipid}: 1,2-dioleoyl-sn-glycero-3-phosphocholine (DOPC) and 1,2-dipalmitoyl-sn-glycero-3-phosphocholine (DPPC) in a $1:2$ molar ratio, with $25$ mol\% cholesterol. At room temperature, this mixture separates into liquid-ordered ($l_o$) and liquid-disordered ($l_d$) phases with substantially different surface viscosities.  We represent the liposome by a stationary sphere $\G$ of radius $R=5\,\mu\mathrm{m}$ and let $c$ denote the local area fraction of the liquid-ordered phase.

We adopt the material parameters tabulated in~\cite{WangPalzhanovQuainiOlshanskiiMajd2022lipid} for this composition at $15^\circ\mathrm{C}$, taking midpoint values where ranges are reported:
\begin{equation}\label{eq:mat-params}
    \begin{aligned}
        \eta_{lo} &= 6.5\times10^{-8}\,\mathrm{Pa\,s\,m},
        &\eta_{ld} &= 0.3\times10^{-8}\,\mathrm{Pa\,s\,m},\\
        D &= 10^{-5}\,\mathrm{cm^2\,s^{-1}},
        &a_D &= 0.71.
    \end{aligned}
\end{equation}
Here $D$ is the diffusivity parameter and $a_D$ is the target area fraction of the liquid-ordered phase. The surface densities reported there correspond to $\rho_{lo}\approx2.44\times10^{-6}\,\mathrm{kg\,m^{-2}}$ and $\rho_{ld}\approx1.95\times10^{-6}\,\mathrm{kg\,m^{-2}}$.

These parameter values support our modeling assumption of  neglecting membrane inertia. Indeed, using the diffusive velocity scale $U=D/R$, we obtain
\[
    \mathrm{Re}
    =\frac{\rho U R}{\eta}
    =\frac{\rho D}{\eta}
    \le
    \frac{2.44\times10^{-6}\cdot10^{-9}}
         {0.3\times10^{-8}}
    \approx8.1\times10^{-7}.
\]
We therefore use a Stokes--Cahn--Hilliard approximation of the surface Abels--Garcke--Gr\"un-type model employed in~\cite{WangPalzhanovQuainiOlshanskiiMajd2022lipid}. In the formal small-Reynolds-number limit, both velocity inertia and the density-dependent transport of momentum by phase diffusion are neglected.

Specifically, we set
\[
    \eta(c)=\eta_{lo}c+\eta_{ld}(1-c),
    \qquad
    M(c)=D c(1-c),
    \qquad
    f(c)=\tfrac14c^2(1-c)^2,
\]
where the argument of the mobility is clipped to $[0,1]$. With the chemical potential normalized as in~\eqref{eq:sch}, the equations read
\begin{equation}\label{eq:sch-unmatched}
    \begin{aligned}
        -2\bP\divG\bigl(\eta(c)D_\G\bu\bigr)+\nablaG p
        &= -\sigma c\nablaG\mu,\\
        \divG\bu &=0,\\
        \pa_t c+\bu\cdot\nablaG c
        -\divG\bigl(\varepsilon M(c)\nablaG\mu\bigr)
        &=0,\\
        \mu &=\varepsilon^{-1}f'(c)-\varepsilon\DeltaG c.
    \end{aligned}
\end{equation}
The factor $\varepsilon$ in the mobility compensates for the normalization of $\mu$, so that $D$ retains its original meaning. We impose $\bu\perp E$ to select the velocity without a rigid rotational component. The phase-dependent viscosity and degenerate mobility extend the constant-coefficient setting analyzed in this paper.

For the present simulation, we take the capillary coefficient $\sigma=1.5\,\mathrm{pN}$, using the midpoint of the reported line-tension range as its numerical value.  We use $\varepsilon=0.2\,\mu\mathrm{m}$, compared with $0.1\,\mu\mathrm{m}$ in that work, to broaden the diffuse transition for the mesh size $h=\varepsilon$. This is a numerical interface-width parameter. The initial phase field is obtained by assigning Bernoulli random values to the active mesh nodes with the  mean $a_D = 0.71$ chosen based on {thermodynamic properties} of the present lipid composition, cf.~\cite{ZhiliakovWangQuainiOlshanskiiMajd2021}. The resulting finite element field represents an initially mixed membrane with small-scale composition fluctuations.

We nondimensionalize using length $R$, time $t_0=R^2/D=0.025\,\mathrm{s}$, velocity $R/t_0=200\,\mu\mathrm{m\,s^{-1}}$, and viscosity $\eta_{ld}$. This gives
\[
    \hat\varepsilon=\frac{\varepsilon}{R}=0.04,
    \qquad
    \hat\sigma=\frac{\sigma R}{\eta_{ld}D}=2.5,
    \qquad
    \hat\eta(c)=1+
    \left(\frac{\eta_{lo}}{\eta_{ld}}-1\right)c.
\]
Thus $\hat\eta$ ranges from $1$ to approximately $21.7$ for $c\in[0,1]$. With $\hat\mu=\hat\varepsilon^{-1}f'(c)-\hat\varepsilon\hat\Delta c$, the dimensionless mobility is $\hat\varepsilon c(1-c)$, with the same clipping convention.

We discretize~\eqref{eq:sch-unmatched} using the proposed TraceFEM with $\bP_2$--$P_1$--$P_1$--$P_1$ elements for $(\bu,p,c,\mu)$ and adapt the midpoint scheme to the phase-dependent viscosity and mobility. The mesh size is $h=0.04R$, giving $758{,}335$ degrees of freedom. We use adaptive time stepping with Rannacher startup. Figure~\ref{fig:liposome} shows snapshots at $t\in\{0,0.095,0.47,1.88,2.95,15,22,24,72.8\}\,\mathrm{s}$. The liquid-disordered domains ($c\approx0$) coarsen within the liquid-ordered majority phase ($c\approx1$). The velocity fields at $t=0.47$ and $1.88\,\mathrm{s}$ accompany pronounced domain deformation, while the later snapshots show larger, more nearly circular domains and slower evolution. These results illustrate coarsening with the viscosity contrast and degenerate mobility of this composition and qualitatively are consistent with experimental data from~\cite{WangPalzhanovQuainiOlshanskiiMajd2022lipid}. Quantitative comparison with microscopy would require comparing domain statistics over multiple realizations of the initial data.

        \begin{figure}[ht]
            \centering
            \includegraphics[width=0.8\linewidth]{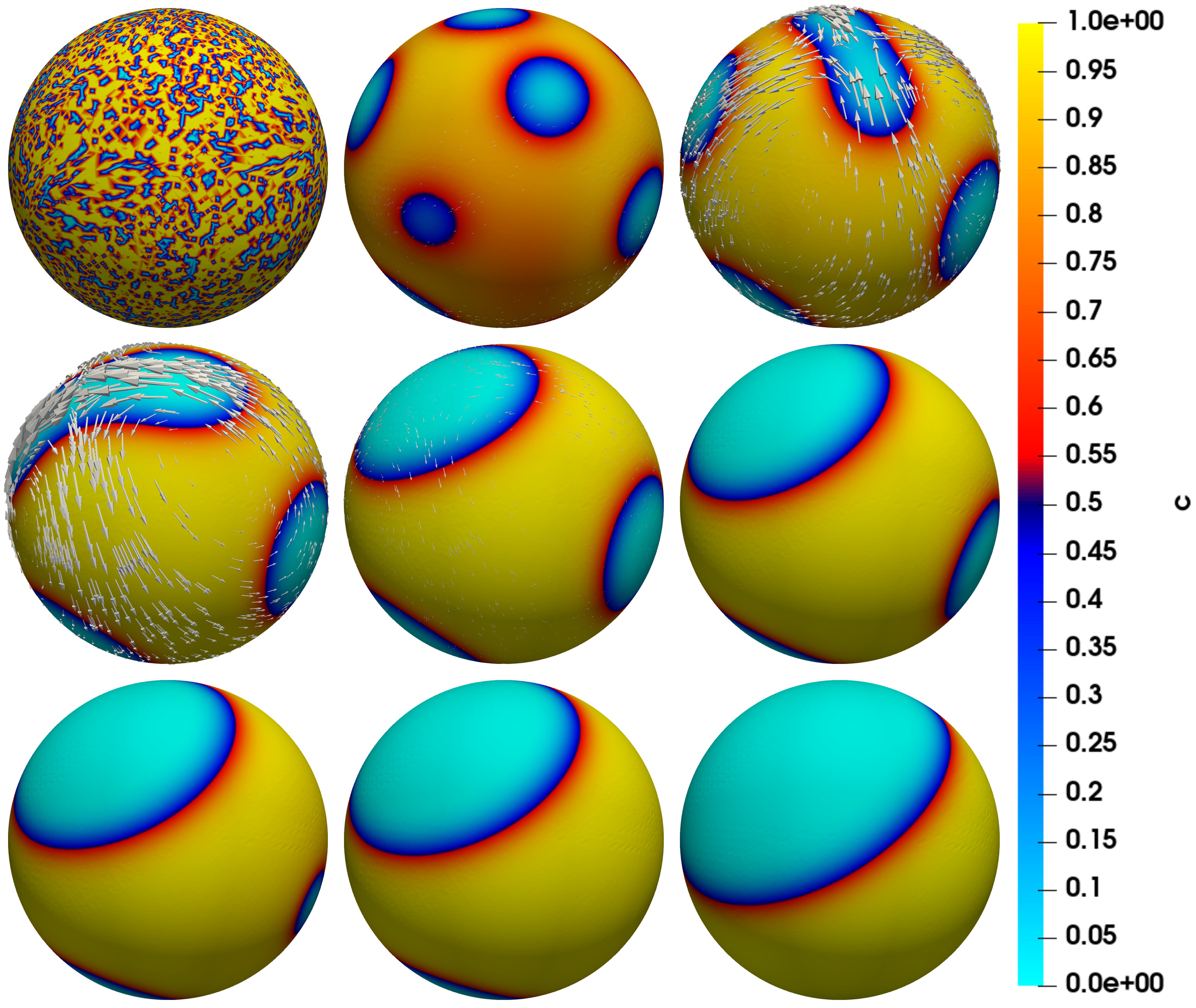}
            \caption{Phase separation on a DOPC--DPPC--cholesterol liposome of composition $25{:}50{:}25$ mol\%. Snapshots are ordered from left to right and top to bottom; colors show the liquid-ordered phase fraction $c$ and arrows indicate the velocity.}
            \label{fig:liposome}
        \end{figure}
        
\section*{Acknowledgments}
 AI tools (Claude Opus 5.5 and ChatGPT-6 Astra) were used in this work for the following purposes:
\begin{itemize}
\item to identify two ingredients of the analysis: testing the phase-field error equation with $f_0'(c_1)-f_0'(c_2)$ in \eqref{eq:chem-pot-ophi} to close the energy estimate, and introducing the centered trilinear form $b$ in \eqref{eq:b-form};
\item to condense the proof of Proposition~\ref{prop:maxreg-L2} on maximal parabolic $L^2$-regularity;
\item to develop code for the numerical experiments;
\item to check the proofs and proofread the manuscript.
\end{itemize}
The authors independently reviewed and validated the mathematical arguments, checked the numerical implementation and results, and take full responsibility for the content of the paper.

\appendix

\section{Auxiliary results} \label{sec:auxiliary-results}
    \subsection{Integration by parts identity} \label{subsec:integration-by-parts}
        The following result was proven in \cite[Lemma 2.1]{JankuhnOlshanskiiReusken2018} for $\G$ of class $C^3$. Below we give its weak version for $\G$ of class $C^2$ following the approach from \cite{BenavidesNochettoShakipov2025-vector}.
        \begin{lemma} \label{lem:bP-divG-nablaGT}
            Let $\G$ be of class $C^2$. Then for all $\bv, \bw \in \bH^1_t(\G)$, $((\nablaG \bv)^T, \nablaG \bw) = (\divG \bv, \divG \bw) - (\bR \bv, \bw)$.
        \end{lemma}
        \begin{proof}
            Let $\bv = (v_i)_{i=1}^{d+1} \in \bC^2(\G)$ (not necessarily tangential), $\uD_j v := (\nablaG v)_j, 1 \leq j \leq d+1$, and $\bn = (n_i)_{i=1}^{d+1}$. It was shown in \cite[Appendix B, proof of Lemma 2.13]{BenavidesNochettoShakipov2025-vector} that $[\bP \divG (\nablaG \bv)^T]_i = \uD_i \uD_l v_l - (\uD_i \uD_l v_m) n_m n_l - \tr(\bB) (\uD_i v_m) n_m$. Here, we use summation over repeated indices. Let $\bw = (w_i)_{i=1}^{d+1} \in \bC^1_t(\G)$. On the one hand, $(\bP \divG (\nablaG \bv)^T, \bw) = - ((\nablaG \bv)^T, \nablaG \bw)$. On the other hand, using integration by parts we obtain
            \[
            \begin{split}
                (\bP \divG (\nablaG \bv)^T, \bw) & = (\nablaG \divG \bv, \bw) - (\uD_i \uD_l v_m, n_m n_l w_i) - (\tr(\bB) (\uD_i v_m) n_m, w_i) \\
                & = - (\divG \bv, \divG \bw) - (\uD_i \uD_l v_m, n_m n_l w_i) - (\tr(\bB) (\uD_i v_m) n_m, w_i).
            \end{split}
            \]
            For the middle term above, we use integration by parts identity $(\uD_i v_i, \varphi) = - (v_i, \uD_i \varphi) + (\tr(\bB) v_i \varphi, n_i)$ together with $w_i n_i = 0$ and $n_l \uD_l v_m = 0$ to get
            \[
            \begin{split}
                (\uD_i \uD_l v_m, n_m n_l w_i) & = - (\uD_l v_m, \uD_i(n_m n_l w_i)) + (\tr(\bB) (\uD_l v_m) n_m n_l w_i, n_i) \\
                & = - (\uD_l v_m, \uD_i(n_m n_l w_i)) \\
                & = - (\uD_l v_m, (\uD_i n_m) n_l w_i) - (\uD_l v_m, n_m (\uD_i n_l) w_i) - (\uD_l v_m, n_m n_l \uD_i w_i) \\
                & = - (\uD_l v_m, n_m (\uD_i n_l) w_i) = - (\uD_l v_m, n_m B_{il} w_i),
            \end{split}
            \]
            where in the end we used $\bB = \nablaG \bn$. Putting it all together yields
            \[
            ((\nablaG \bv)^T, \nablaG \bw) = (\divG \bv, \divG \bw) - (\uD_l v_m, n_m B_{il} w_i) + (\tr(\bB) (\uD_i v_m) n_m, w_i).
            \]
            Note that the above equality {involves only first derivatives} of $\bv, \bw$, hence it can be extended to $\bv \in \bC^1(\G)$. We now assume $\bv \in \bC^1_t(\G)$ and use $B_{ij} v_j = - (\uD_i v_j) n_j$ from \cite[Eq. (2.6)]{BenavidesNochettoShakipov2025-vector} valid for tangential fields. With this, we obtain $((\nablaG \bv)^T, \nablaG \bw) = (\divG \bv, \divG \bw) - ([\tr(\bB) \bB - \bB \bB] \bv, \bw)$. Recognizing that $\bR = \tr(\bB) \bB - \bB \bB$ and using density finishes the proof.
        \end{proof}

    \subsection{Maximal parabolic \texorpdfstring{$L^2$}{L2}-regularity} \label{subsec:maximal-regularity}
        We start with the following elementary result.
        \begin{proposition}[well-posedness of fourth-order parabolic problem] \label{prop:parab}
            Let $\G$ be of class $C^2$. Then $X_0 := L^2(I; H^2(\G)) \cap H^1(I; H^2(\G)') \hookrightarrow C^0([0,T]; L^2(\G))$. Moreover, $X_0 \ni u \mapsto (\pa_t u + \varrho \DeltaG^2 u, u(0)) \in L^2(I; H^2(\G)') \times L^2(\G)$ is a linear homeomorphism for each $\varrho > 0$, and the stability constants depend only on $\varrho, \G$ and $T$.
        \end{proposition}
        \begin{proof}
            First note that for $\G$ of class $C^2$, $H^2(\G)$ is dense in $L^2(\G)$. Thus, setting $V = H^2(\G), H = L^2(\G)$ yields a Gelfand triple $V \hookrightarrow H \cong H' \hookrightarrow V'$, whence the classical Lions--Magenes lemma yields the asserted embedding. Then {Lions' theorem} together with $H^2(\G)$-regularity for $\DeltaG$ guarantees the second result. This finishes the proof.
        \end{proof}

        Let $\{(\lambda_i, \psi_i)\}_{i=0}^\infty$ be the eigendecomposition of $-\DeltaG$, with $\{\psi_i\}$ being an $L^2(\G)$-orthonormal basis and $\lambda_0 = 0$. We define $\Lambda := (\operatorname{id} - \DeltaG)^{1/2}$ spectrally:
        \begin{equation*}
            \Lambda v := \sum_{i=0}^\infty \sqrt{1 + \lambda_i} (v, \psi_i) \psi_i.
        \end{equation*}
        We next define the following inner products for integer $s \geq 0$: 
        \begin{equation*}
            (u, w)_s := (\Lambda^s u, \Lambda^s w), \qquad \|u\|_s^2 := (u,u)_s.
        \end{equation*}
        For $\G$ of class $C^{m+4}, m \geq 0$, elliptic regularity \cite[Theorem~1.2]{BenavidesNochettoShakipov2026-scalar} shows that $\|\cdot\|_s$ is an equivalent norm on $H^s(\G)$ for $0 \leq s \leq m+4$. By spectral calculus, $\DeltaG$ and $\Lambda$ commute. Together with the $L^2(\G)$-symmetry of $\DeltaG$, this implies $(\DeltaG u, w)_s = (u, \DeltaG w)_s$ for all $u, w \in H^{s+2}(\G), 0 \leq s \leq m+2$.        
        
        This brings us to the following key embedding:
        \begin{equation} \label{eq:lions-magenes}
            X_{m+2} := L^2(I; H^{m+4}(\G)) \cap H^1(I; H^m(\G)) \hookrightarrow C^0([0,T]; H^{m+2}(\G)), \qquad m \geq 0.
        \end{equation}
        To see this, define $V := H^{m+4}(\G)$ and $H := H^{m+2}(\G)$ equipped with the inner product $(\cdot,\cdot)_{m+2}$, and use $V \hookrightarrow H \cong H' \hookrightarrow V' \cong H^m(\G)$, where the first embedding is continuous and dense. The identification $V' = H^{m+4}(\G)' \cong H^m(\G)$ follows from elliptic regularity for $\G$ of class $C^{m+4}$ in the following way:
        \begin{equation*}
            \Lambda^4 = (\operatorname{id} - \DeltaG)^2 : H^{m+4}(\G) \to H^m(\G) \quad \text{ and } \quad \Lambda^{2m} = (\operatorname{id} - \DeltaG)^m : H^m(\G) \to H^m(\G)'
        \end{equation*}
        are linear homeomorphisms. Concatenating the two and dualizing yields the asserted identification $V' \cong H^m(\G)$. Explicitly, $g \in H^m(\G)$ is identified with the functional $v \mapsto (g, \Lambda^4 v)_m$ (for $v \in V = H^{m+4}(\G)$), which for $g \in H = H^{m+2}(\G)$ equals $(g, v)_{m+2}$, so the identification is compatible with the embedding $H \hookrightarrow V'$. Embedding \eqref{eq:lions-magenes} then follows from the classical Lions--Magenes lemma.
    
        \begin{proof}[Proof of Proposition \ref{prop:maxreg-L2}]
            Recall that $m \geq 0$ is an integer. We work in the Hilbert space $H^m(\G)$ equipped with the inner product $(\cdot,\cdot)_m$. We define the following energy functional on $H^{m+2}(\G)$:
            \begin{equation*}
                \varphi(u) := \frac{\varrho}{2} \|\DeltaG u\|_m^2.
            \end{equation*}
            Extended by $+\infty$ to $H^m(\G)$, $\varphi$ is convex, lower semicontinuous and nonnegative. For all $u,v \in H^{m+2}(\G)$, we have
            \begin{equation*}
                \frac{\d}{\d s} \varphi(u + s v) \big|_{s=0} = \varrho (\DeltaG u, \DeltaG v)_m.
            \end{equation*}
            Therefore, by the symmetry of $\DeltaG$ with respect to $(\cdot,\cdot)_m$, $\pa \varphi = \varrho \DeltaG^2$ with $D(\pa\varphi) = \{u \in H^{m+2}(\G) : \DeltaG^2 u \in H^m(\G)\}$, which coincides with $H^{m+4}(\G)$ by elliptic regularity. Given $f \in L^2(I; H^m(\G))$ and $u_0 \in H^{m+2}(\G) = D(\varphi)$, \cite[Th\'eor\`emes~3.4 and~3.6]{Brezis1973} yield a unique solution of the abstract parabolic problem
            \begin{equation*}
                u' + \pa \varphi (u) = f, \qquad u(0) = u_0,
            \end{equation*}
            with $\|u'\|_{L^2(I; H^m(\G))} \lesssim \|f\|_{L^2(I; H^m(\G))} + \sqrt{\varphi(u_0)}$. Hence 
            \begin{equation*}
                \pa \varphi (u) = \varrho \DeltaG^2 u = f - u' \in L^2(I; H^m(\G)),
            \end{equation*}
            and elliptic regularity gives $u \in L^2(I; H^{m+4}(\G))$ with $\|u\|_{X_{m+2}} \leq C(\|f\|_{L^2(I; H^m(\G))} + \|u_0\|_{H^{m+2}(\G)})$, i.e., the inverse map is bounded. Uniqueness in $X_{m+2}$ follows from Proposition \ref{prop:parab}, since $X_{m+2} \subseteq X_0$. Finally, the map itself is bounded thanks to the embedding $X_{m+2} \hookrightarrow C^0([0,T]; H^{m+2}(\G))$.
        \end{proof}
        
\bibliographystyle{ieeetr}
\bibliography{references}

\end{document}